\documentclass[11pt]{amsart}
\usepackage{amscd, amssymb,latexsym, amsmath, amscd, amsmath, comment, bbm, amsfonts}
\usepackage{tikz-cd}
\usepackage{extarrows}
\usepackage{mathtools}
\usepackage{hyperref}
\usepackage[backend=biber, style=alphabetic]{biblatex}
\hypersetup{%
        colorlinks,
        breaklinks=true, 
        plainpages=false,%
        citecolor=blue,
        linkcolor=blue,
        urlcolor=blue,
        bookmarksopen=true,%
        bookmarksnumbered=false,%
        bookmarksdepth=5%
}

\newcommand{\Q}{\mathbb{Q}}

\newcommand{\Z}{\mathbb{Z}}

\newcommand{\F}{\mathbb{F}}

\newcommand{\End}{{\rm End}\,}

\newcommand{\Fq}{\mathbb{F}_q}
\newcommand{\im}{\rm{Im}\>}
\newcommand{\Ker}{\rm{Ker}\>}
\newcommand{\cH}{\mathcal{H}}

\newcommand\restr[2]{{
  \left.\kern-\nulldelimiterspace 
  #1 
  \littletaller 
  \right|_{#2} 
  }}
\newcommand{\littletaller}{\mathchoice{\vphantom{\big|}}{}{}{}}

\def\diag{{\rm diag}}

\newcommand{\IZind}{\mathrm{ind}_{IZ}^G\>}
\newcommand{\KZind}{\mathrm{ind}_{KZ}^G\>}

\newcommand{\IIZind}{\mathrm{ind}_{I(1)Z}^G\>}
\newcommand{\IIZIZind}{\mathrm{ind}_{I(1)Z}^{IZ}\>}
\newcommand{\id}{\mathrm{id}}
\newcommand{\Id}{\mathrm{Id}}

\begin{document}
\newtheorem{theorem}{Theorem}[section] 
\newtheorem{lemma}[theorem]{Lemma} 
\newtheorem{proposition}[theorem]{Proposition} 
\newtheorem{corollary}[theorem]{Corollary} 
\newtheorem{remark}[theorem]{Remark}
\newtheorem{conjecture}[theorem]{Conjecture}
\newtheorem{question}[theorem]{Question}
\newtheorem{example}[theorem]{Example}
\numberwithin{equation}{section}




\title{On the $I(1)$-invariants: Non-abelian Hecke algebra case}

\author[Anand Chitrao]{Anand Chitrao}
\author[Arindam Jana]{Arindam Jana}
\author[Asfak Soneji]{Asfak Soneji}
\address{Department of Mathematical Sciences, UNIST, Ulsan - 44919, Republic of Korea}
\email{anand@unist.ac.kr}
\address{Kerala School of Mathematics, Kozhikode, Kerala - 673571, India}
\email{arindam@ksom.res.in}
\address{Department of Basic Sciences, IITRAM, Ahmedabad - 380026, India}
\email{sonejiasfak96@gmail.com}

\begin{abstract}

    Let $F$ be a finite extension of $\mathbb{Q}_p$. The so-called supersingular representations are the basic building blocks in the theory of mod $p$ representations of ${\rm GL}_2(F)$. 
	The space of pro-$p$-Iwahori invariants of a universal module played a crucial role in the construction of the supersingular representations of ${\rm GL}_2(\mathbb{Q}_p)$. 
	In this paper, 
	we give an explicit description of the pro-$p$-Iwahori invariants of the universal module $\pi_r$ for $r = 0, q - 1$ using the Iwahori-Hecke model.
	We also determine the action of the pro-$p$-Iwahori-Hecke algebra on these
	newly found invariants. 
    As an application, we recover $\pi_r$ functorially from its space of $I(1)$-invariants
    and extend a theorem of Ollivier for any totally ramified extension of $\mathbb{Q}_p$ other than itself.
\end{abstract}

\maketitle


\section{Notation and conventions}
For a prime $p$, let $\F_q=\F_{p^{f}}$ be a degree $f$ extension of the finite field $\F_p$. Let $\mathbb{\overline{F}}_{p}$ denote an algebraic closure of $\F_p$. We fix an embedding $\F_q \hookrightarrow \mathbb{\overline{F}}_{p}$. Let $F$ be a finite extension of $\mathbb{Q}_p$ with ring of integers $\mathcal{O}$, uniformizer $\varpi$, and residue field $\Fq$. Let $G$  be the group ${\rm{GL}_2}(F)$ and $Z$ be its center. Let $K$ be the maximal compact subgroup $\rm{GL_2}(\mathcal{O})$. Let $I$ denote the Iwahori subgroup of upper triangular matrices mod $\varpi$, and $I(1)$ be the pro-$p$-Iwahori subgroup of $K$ consisting of matrices that are upper triangular unipotent mod $\varpi$. For $x \in \mathbb F_q$, we denote its multiplicative representative in $\mathcal O$ by $[x]$. Let
\[
    H := \left\{\left.\begin{pmatrix}[\lambda] & 0 \\ 0 & [\mu]\end{pmatrix} \right\vert \lambda, \mu \in \Fq^\times\right\}.
\]
Set $I_0=\{0\}$, and for any integer $n \geq 1$, let
\[I_n=\left\{[{\mu}_0]+[{\mu}_1]\varpi+\dots+[\mu_{n-1}]{\varpi}^{n-1}\mid
{\mu}_i\in \mathbb{F}_q\right\} \subset \mathcal{O}.\]

If $0\leq m\leq n,$ let $[\cdot]_m:I_n\rightarrow I_m$ be the truncation map defined by
\[\sum\limits_{\substack{i=0}}^{n-1}[\lambda_i]\varpi^i\mapsto \sum\limits_{\substack{i=0}}^{m-1}[\lambda_i]\varpi^i.\]  
Let us denote
\[\alpha=\left(\begin{array}{cc}
1 & 0 \\
0 & \varpi
\end{array}\right),~\beta=\left(\begin{array}{cc}
0      & 1 \\
\varpi & 0
\end{array}\right),~ w=\left(\begin{array}{cc}
0 & 1 \\
1 & 0
\end{array}\right),~ n_s=\left(\begin{array}{cc}
0 & -1 \\
1 & 0
\end{array}\right),\]
and observe that $\beta=\alpha w$ normalizes $I(1)$. For any integer $n \geq 1$, we denote
\begin{align*}
s_n^k &=\sum\limits_{\substack{\mu\in I_n}}\mu_{n-1}^k\left[\left(\begin{array}{cc}
\varpi^n & \mu \\
0    &  1
\end{array}\right),1\right],\\
t_n^k &=\sum\limits_{\substack{\mu\in I_n}}\mu_{n-1}^k\left[\left(\begin{array}{cc}
\varpi^{n-1} & [\mu]_{n-1} \\
0    &  1
\end{array}\right) 
\left(\begin{array}{cc}
1& [\mu_{n-1}] \\
0 & 1
\end{array}\right)w,1\right],
\end{align*} 
where $0\leq k\leq q-1$. By convention, we define $0^0=1$.

For any integer $n \geq 0$ and $\lambda\in I_n$, define
\[g_{n,\lambda}^0=\left(\begin{array}{cc}
\varpi^n & \lambda \\
0 & 1
\end{array}\right)~~\text{and}~~ g_{n,\lambda}^1=\left(\begin{array}{cc}
1             & 0 \\
\varpi\lambda & \varpi^{n+1}
\end{array}\right).\]
We have the relations 
\[g_{0,0}^0={\rm Id},~g_{0,0}^1=\alpha,\ \beta g_{n,\lambda}^0=g_{n,\lambda}^1w.\]
Now $G$ acts transitively on the Bruhat-Tits tree of ${\rm SL}_2(F)$ \cite{Ser03}. The vertices are in a $G$-equivariant bijection with $G/{KZ}$ and the oriented edges are in a $G$-equivariant bijection with $G/{IZ}$. We have the explicit Cartan decomposition given by
\[G=\underset{\substack{i \in \{0, 1\} \\ n \geq 0,~ \lambda \in I_n}}\coprod g_{n,\lambda}^i KZ \] 
and a set of coset representatives of $G/IZ$ is given by
\begin{equation}\label{edges}
\left\{  g_{n,\lambda}^0, g_{n,\lambda}^0 
\left(\begin{array}{cc}
1& \mu \\
0 & 1
\end{array}\right)w, g_{n,\lambda}^1w, g_{n,\lambda}^1w 
\left(\begin{array}{cc}
1& \mu \\
0 & 1
\end{array}\right)w
\right\}_{n \geq 0, \lambda \in I_n, \mu \in I_1}.
\end{equation}

We choose the following orientation of the Bruhat-Tits tree for $\mathrm{SL}_2(F)$:
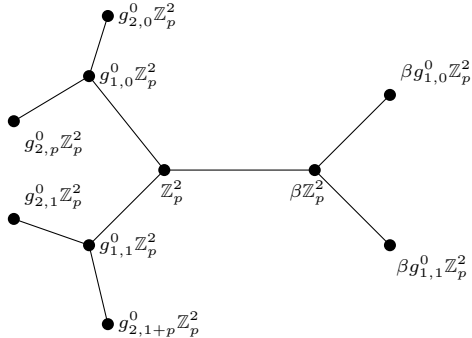
\begin{center}\label{Bruhat-Tits tree}
    \[
        \begin{tikzpicture}
            \draw[black] (-1, 0) -- (1, 0);
            \draw[black] (1, 0) -- (2, 1);
            \draw[black] (1, 0) -- (2, -1);
            \draw[black] (-2, 1.25) -- (-1, 0);
            \draw[black] (-2, -1) -- (-1, 0);
            \draw[black] (-2, 1.25) -- (-1.75, 2.05);
            \draw[black] (-2, 1.25) -- (-3, 0.65);
            \draw[black] (-2, -1) -- (-3, -0.65);
            \draw[black] (-2, -1) -- (-1.75, -2.05);
            \fill[black] (-1, 0) circle (0.08cm);
            \filldraw[black] (-1, 0) circle node[font = \tiny, anchor=north]{\>\> $\Z_p^2$};
            \fill[black] (1, 0) circle (0.08cm);
            \filldraw[black] (1, 0) circle node[font = \tiny, anchor=north]{\!\!\!\!\!\! $\beta\Z_p^2$};
            \fill[black] (2, 1) circle (0.08cm);
            \filldraw[black] (2, 1) circle node[font = \tiny, anchor=south west]{\!\! $\beta g^0_{1, 0}\Z_p^2$};
            \fill[black] (2, -1) circle (0.08cm);
            \filldraw[black] (2, -1) circle node[font = \tiny, anchor=north]{\qquad \qquad $\beta g^0_{1, 1}\Z_p^2$};
            \fill[black] (-2, 1.25) circle (0.08cm);
            \filldraw[black] (-2, 1.25) circle node[font = \tiny, anchor=west]{$g^0_{1, 0}\Z_p^2$};
            \fill[black] (-2, -1) circle (0.08cm);
            \filldraw[black] (-2, -1) circle node[font = \tiny, anchor=west]{$g^0_{1, 1}\Z_p^2$};
            \fill[black] (-1.75, 2.05) circle (0.08cm);
            \filldraw[black] (-1.75, 2.05) circle node[font = \tiny, anchor=west]{$g^0_{2, 0}\Z_p^2$};
            \fill[black] (-3, 0.65) circle (0.08cm);
            \filldraw[black] (-3, 0.65) circle node[font = \tiny, anchor=north west]{$g^0_{2, p}\Z_p^2$};
            \fill[black] (-3, -0.65) circle (0.08cm);
            \filldraw[black] (-3, -0.65) circle node[font = \tiny, anchor=south west]{$g^0_{2, 1}\Z_p^2$};
            \fill[black] (-1.75, -2.05) circle (0.08cm);
            \filldraw[black] (-1.75, -2.05) circle node[font = \tiny, anchor=west]{$g^0_{2, 1 + p}\Z_p^2$};
        \end{tikzpicture}
        \]
    Case $q = 2$.
    \end{center}
    So, the phrase ``left side of the tree'' refers to the vertices associated with the lattices $g^0_{n, \mu}\Z_p^2$ for $n \geq 1$ and $\mu \in I_n$. The vertex associated with the lattice $\Z_p^2$ is called the ``root vertex''. The phrase ``right side of the tree'' means the complement of the union of the left side of the tree and the root vertex.
    
Let $B(m)$ denote the ball of radius $m$ in the tree of ${\rm SL}_2(F)$ with center at the root vertex. Explicitly, it consists of linear combinations of vectors in the sets
\[B^0(m) = \left\{[g_{n,\mu}^0,1], \left[g_{n-1,[\mu]_{n-1}}^0 
\left(\begin{array}{cc}
1& [\mu_{n-1}] \\
0 & 1
\end{array}\right)w,1\right]= [g_{n,\mu}^0 \beta, 1] \right\}_{n \leq m, \mu \in I_n},\]
and
\[B^1(m) =\left\{[g_{n-1,[\mu]_{n-1}}^1w,1], \left[g_{n-2,[\mu]_{n-2}}^1w 
\left(\begin{array}{cc}
1& [\mu_{n-2}] \\
0 & 1
\end{array}\right)w,1\right] = [g_{n-1,[\mu]_{n-1}}^1w \beta, 1]\right\}_{n \leq m, \mu \in I_n}.\]
We say that an edge is in $B(n)$ for some $n \geq 0$ if its characteristic function is in $B(n)$.

Let $\cH$ denote the pro-$p$-Iwahori-Hecke algebra: $\cH = \End_{G}(\mathrm{ind}_{I(1)Z}^{G}\>\mathbbm{1})$. A description of $\cH$ can be found, e.g., in \cite[Chapter 2]{Pas04} or \cite{Vig04}. We will be working with representations having the trivial central character. The Hecke algebra $\cH$ is generated by the operators $T_{\beta}, T_{n_s}$, and $e_{\chi} = |H|^{-1}\sum_{h \in H}\chi(h)T_h$, where $\chi$ runs through all characters $H \to \overline{\F}_p^{\times}$. We denote the character $\chi\begin{pmatrix}[\lambda] & 0 \\ 0 & [\mu]\end{pmatrix} = \lambda^r \mu^s$ of $H$ by $a^rd^s$. If $\chi = a^r d^s$, then let $\chi^{w} = a^s d^r$.

\section{Introduction}
The study of mod $p$ representations of $G$ began with the pioneering works of Barthel and Livn\'e in \cite{BL94, BL95}. The authors classified the smooth irreducible representations of $G$ with a central character over characteristic $p$ fields into four disjoint families: characters, twists of the Steinberg representation, principal series representations, and supersingular representations. They further showed that all of these representations can be realized as quotients of the compactly induced representation $\KZind \sigma$  where $\sigma$ is a smooth irreducible representation of $K$ with the action extended to $KZ$ by declaring that $\varpi \in Z$ acts trivially. 

The first three were realized as explicit quotients of $\KZind \sigma_r / {\im} (T - \lambda)$ for some $0 \leq r \leq q - 1$ and $\lambda \in \overline{\F}_p^{\times}$. Here $T$ is the standard spherical Hecke operator defined in \cite[Section 3.1]{BL94}.

The supersingular representations, however, were only realized as non-explicit irreducible quotients of the universal module
\[
    \pi_r = \frac{\KZind \sigma_r}{{\im} T}.
\]
The finding of an explicit supersingular quotient seems to be a much harder question, although some abstract constructions are available in \cite{BP12, Le19, GS20, She22, GLS23, Sch23}.
The first progress towards the explicit understanding of the supersingular representations was made by Breuil in \cite{Bre03}, where he showed that the representations $\pi_r$ are themselves irreducible when $F = \Q_p$. This was done by showing that the space of $I(1)$-invariants in $\pi_r$ is two-dimensional. We note in passing that the subspace of $\pi_r$ invariant under the principal congruence subgroups can be found in \cite{Mor13} and \cite{AB13}. For totally ramified $F/ \mathbb{Q}_p$, an irreducibility criterion for certain quotients of $\KZind \sigma$ is given in \cite{Sch11}.

The space of $I(1)$-invariants of $\pi_r$ was determined by Schein in \cite{Sch14} for $F/\Q_p$ totally ramified. This result was extended by Hendel in \cite{Hen19} to include all $F/\Q_p$ and all $r$ for which the $p$-adic digits satisfy $2 < r_j < p - 3$. In \cite{Jan23}, the second author then extended Hendel's result to the range $1 < r_j < p - 1$ when $F/\Q_p$ is not totally ramified. This uses the Iwahori-Hecke model for $\pi_r$ first introduced by Anandavardhanan and Borisagar in \cite{AB15} for $F = \Q_p$ and $r \neq 0, p - 1$. The Iwahori-Hecke model was extended to arbitrary $F/\Q_p$ by Anandavardhanan and the second author in \cite{AJ21}, where they used it to study a further quotient of the universal module $\pi_r$ when the $p$-adic digits of $r$ satisfy $0 < r_j < p-1$ for $F/ \mathbb{Q}_p$ not totally ramified and $2 < r < p-3$ for $F/ \mathbb{Q}_p$ totally ramified. Note that in  \cite{AJ21, Jan23} the Iwahori-Hecke algebra was commutative.

However when $r = 0, q - 1$, the Iwahori-Hecke algebra ${\rm End}_G(\IZind \mathbbm{1})$ becomes non-commutative. Indeed, we have 
\[
  {\rm End}_G(\IZind \mathbbm{1}) \simeq \dfrac{\overline{\F}_p[T_{1, 0}, T_{1, 2}]}{(T_{1,0}^2-\Id,\  T_{1, 2} T_{1, 0} T_{1, 2} + T_{1, 2})},
\]
where $T_{1, 0}$ and $T_{1, 2}$ denote the Iwahori-Hecke operators.

Therefore the computation of $I(1)$-invariants of $\pi_r$ for $r=0, q-1$ needs a slightly different analysis which is carried out in Sections \ref{preli-lemmas} and \ref{pro-p-invarints}. An Iwahori-Hecke model was obtained in \cite[Theorem 3.13]{Chi25} which was crucial in computing the reduction of certain $p$-adic Galois representations of the absolute Galois group of $\mathbb{Q}_p$ \cite{CG24, CG26}. For arbitrary local fields $F$ an Iwahori-Hecke model of $\pi_r$ with $r = 0, q-1$ was obtained in \cite[Theorem 3]{CS25} using results from \cite{GJ25}. Using this model, we give an explicit description of the subspace of $I(1)$-invariants of $\pi_r$ for $r = 0, q-1$.

Next we show an application of the $I(1)$-invariant space obtained in Theorem \ref{pro-p-iwahori-invariants-intro}. In \cite{Oll09}, it is shown that the category of $\overline{\F}_p$-representations of $G$ generated by the $I(1)$-invariants is equivalent to that of the simple $\cH$-modules for $F = \mathbb{Q}_p$, but not for $F/\mathbb{Q}_p$ with residue cardinality $q>p$. In fact for such $F$, there exists a simple supersingular $\cH$-module $M$ for which $( M^{I(1)} \otimes_{\cH} \IIZind \mathbbm{1})^{I(1)}$ is infinite dimensional \cite[Theorem 1.2 (b)]{Oll09}. The subspace of $I(1)$-invariants of $\pi_r$ described in Theorem \ref{pro-p-iwahori-invariants-intro} will be very handy to check this statement by taking $M = \langle [\id, 1], [\beta, 1] \rangle \cH$, and it further extends \cite[Theorem 1.2 (b)]{Oll09} for $F/\mathbb{Q}_p$ totally ramified and $F \neq \mathbb{Q}_p$. Finally we recover $\pi_0$ as a direct summand of the representation obtained by applying the reverse functor studied by Vign\'eras, Ollivier, and others to $\pi_0^{I(1)}$. 

These functors have been studied by others for groups other than ${\rm GL}_2$. In particular, Koziol has shown that the functor of $I(1)$-invariants induces an equivalence of categories for ${\rm SL}_2(\mathbb{Q}_p)$ with $p>2$ \cite[Corollary 5.3]{Koz16}.

Now we state the theorems of this paper. Assume, for the finite extension $F$ of $\mathbb{Q}_p$ that $ef>1$ and $q>3$. Note that the cases $q=2, 3$ were also excluded from the computation of $I(1)$-invariants in \cite{Hen19, AJ21, Jan23} (see the discussion before Remark \ref{q = 2, 3}).

\begin{theorem}\label{pro-p-iwahori-invariants-intro}
 Let 
    \[ 
        \tau := \frac{\IZind \mathbbm{1}}{({\im} T_{1,2} , {\Ker} T_{1,2})}
    \] 
    be the Iwahori-Hecke model of the universal supersingular representation $\pi_0$ of $G$. Then a basis for the set of $I(1)$-invariants of $\tau$ is
    \[
        \left\{[\id, 1], [\beta, 1], s_n^{p^l}, \beta s_n^{p^l}\right\}_{\substack{n \geq 2 \\ 0 \leq l \leq f - 1}}.
    \] 
\end{theorem}

As an application, we show that $\pi_0$ is indecomposable, which follows from the following theorem.

 \begin{theorem}\label{endo-algebra-intro}
        The endomorphism algebra $\End_G(\pi_0) = \overline{\F}_p$.
 \end{theorem}

We compute the explicit action of the pro-$p$-Iwahori-Hecke algebra $\mathcal{H}$ on the space of $I(1)$-invariants $\pi_0^{I(1)}$ in $\pi_0$ (see \S\ref{action-on-I(1)-invariants}). These actions are essential in showing that $\pi_0$ can be realized as a direct summand of the representation $\pi_0^{I(1)} \otimes_{\cH} \IIZind \mathbbm{1}$. Indeed, we have
\begin{theorem}\label{summand-in-reverse-vigneras-functor-intro}
   Let $\pi_0$ denote the universal supersingular representation of $G$. Then there exists a representation $M$ of $G$ such that
   \[
   \pi_0^{I(1)} \otimes_{\cH} \IIZind \mathbbm{1} \simeq \pi_0 \oplus M.
   \]
\end{theorem}

\section{Preliminary Lemmas}\label{preli-lemmas}
In this section, we recall some generalities regarding the Hecke operators. It also consists of some basic results which will be crucial in proving Theorems \ref{pro-p-iwahori-invariants-intro} and \ref{endo-algebra-intro}. We begin with the following elementary lemma.
\begin{lemma}\label{Every fun is a poly}
    Every function $f: \Fq^n \to \Fq$ comes from a polynomial in $n$ variables with degree in each variable less than or equal to $q - 1$.
\end{lemma}
\begin{proof}
    There are $q^{q^n}$ number of functions from $\Fq^n \to \Fq$. Also, there are $q^{q^n}$ number of polynomials in $n$ variables over $\Fq$ of degrees less than or equal to $q - 1$ in each variable. Therefore, the obvious injection from the space of such polynomials to the space of functions from $\Fq^n \to \Fq$ is also surjective.
\end{proof}
As mentioned in the introduction, when $r=0$ and $q-1$, the Iwahori-Hecke algebra is non-commutative. It is generated by the Iwahori-Hecke operators $T_{1,0}$ and $T_{1,2}$, which satisfy the following relations: 
$$T_{1,0}^2 = \Id \quad \text{and} \quad T_{1,2}T_{1,0}T_{1,2} = -T_{1,2}.$$ 
We recall the formulas for the Iwahori-Hecke operators \cite{BL95}:
\begin{eqnarray}\label{Formulas T10 and T12}
        T_{1, 0}[g, 1] = [g\beta, 1] \text{ and } T_{1, 2}[g, 1] = \sum_{\lambda \in I_1}\left[g\begin{pmatrix}1 & 0 \\ \varpi\lambda & \varpi\end{pmatrix}, 1\right].
    \end{eqnarray}
The operator $T_{-1,0}$ satisfies the following relation (see \cite[Lemma 8(4)]{BL95}):  
\begin{eqnarray}\label{T-10 in the non-commutative case}
        T_{-1, 0} = T_{1, 0}T_{1, 2}T_{1, 0},
    \end{eqnarray}
and its formula is given by:
\begin{eqnarray}\label{Formulas for T-10 and T12}
T_{-1, 0}[g, 1] = \sum_{\lambda \in I_1}\left[g\begin{pmatrix}\varpi & \lambda \\ 0 & 1\end{pmatrix}, 1\right].
\end{eqnarray}
We state the following lemma from \cite[Lemma 4]{CS25}:
\begin{lemma}\label{L1}
  For the operators $T_{-1, 0}$ and $T_{1, 2}$, We have
  \begin{enumerate}
      \item[(i)] ${\im} T_{1,2} = {\Ker}(T_{1,0}+T_{-1,0})$,
      \item[(ii)] ${\Ker} T_{1,2} = {\im}(T_{1,0}+T_{-1,0})$.
  \end{enumerate}
\end{lemma}

The next lemma gives a necessary condition for a function in the induced space $\IZind \sigma_0$ to be in the image of the Hecke operator $T$.
    
\begin{lemma}\label{functions are 0}
    Let $f$ be a function in $\KZind \sigma_{0}$ supported on the vertices in $B(n)\setminus B(n - 1)$
    for some integer $n \geq 1$. Then, for $f$ to belong to ${\im} T$, it must have the same value on all vertices of radius $n$ having a common neighbour of radius $n - 1$.
\end{lemma}
\begin{proof}
    Let $f$ be a function with support contained in $B(n)$ but not in $B(n - 1)$ for some $n \geq 1$. Let $f' \in \KZind \sigma_0$ be a function such that $f = Tf'$. By the definition of $T$, one can check that the support of $f'$ has to be contained in $B(n - 1)$. Moreover, the value of $Tf'$ at a vertex $v \in B(n)\setminus B(n - 1)$ is equal to the value of $f'$ at the vertex in $B(n - 1)$ adjacent to $v$. This proves the lemma.
\end{proof}
\subsection{Reduction Lemmas} 
In this section we analyze the reduction of the functions in the tree of ${\rm SL}_2(F)$ modulo $({\im} T_{1, 2}, {\Ker} T_{1, 2})$. We begin by showing that, modulo the image of $T_{1, 2}$, a function supported on a vector in $B(n)$ directed away from the root vertex reduces to a function supported on a vector in $B(n-1)$ directed towards the root vertex.
\begin{lemma}\label{tnk like terms}
 Let $n \geq 1$ and $0 \leq k \leq q-1$.
Then, modulo ${\im} T_{1, 2}$, we have
\[
    \sum_{\mu \in I_n}c_{[\mu]_{n - 1}}\mu_{n - 1}^k\left[g^0_{n - 1, [\mu]_{n - 1}}\!\!\begin{pmatrix}1 & [\mu_{n - 1}] \\ 0 & 1\end{pmatrix}w, 1\right] = \begin{cases}
        0 & \!\!\text{ if } k \neq q - 1, \\
        -\displaystyle \!\!\!\!\!\!\sum_{[\mu]_{n-1} \in I_{n-1}}\!\!\!\!\!\!c_{[\mu]_{n - 1}}\left[g^0_{n-1, [\mu]_{n-1}}, 1\right] &\!\! \text{ if } k = q - 1,
    \end{cases}
\]
where $c_{[\mu]_{n - 1}}:=c(\mu_{0}, \ldots, \mu_{n - 2})$.

\end{lemma}
\begin{proof}
First, let us assume that $k$ is not equal to $q-1$. Note that
\begin{equation}\label{sum in t_1 family}
   \sum_{\mu \in I_n}c_{[\mu]_{n - 1}}\mu_{n - 1}^k\left[g^0_{n - 1, [\mu]_{n - 1}}\begin{pmatrix}1 & [\mu_{n - 1}] \\ 0 & 1\end{pmatrix}w, 1\right] = \sum\limits_{\substack{[\mu]_{n-1}\in I_{n-1}}}c_{[\mu]_{n - 1}}g^0_{n - 1, [\mu]_{n - 1}} t_1^k. 
\end{equation}

Thus, it is sufficient to prove that $t_1^k \in {\im}{T_{1,2}}$ for $k \neq q-1$. 
Applying $T_{1, 2}$ to $s_1^{k}$, we get
    \begin{eqnarray*}
        T_{1, 2}(s_1^{k}) & = & \sum_{\mu \in I_1}\sum_{\lambda \in I_1}\mu^{k}\left[\begin{pmatrix}\varpi & \mu \\ 0 & 1\end{pmatrix}\begin{pmatrix}1 & 0 \\ \varpi \lambda & \varpi\end{pmatrix}, 1\right] \\
         & = & \sum_{\mu \in I_1}\sum_{\lambda \in I_1}\mu^k\left[\begin{pmatrix}\varpi + \varpi \lambda \mu & \varpi \mu \\ \varpi \lambda & \varpi\end{pmatrix}, 1\right] \\
        & = & \sum_{\mu \in I_1}\sum_{\lambda \in I_1} \mu^k \left[\begin{pmatrix}1 + \lambda \mu & \mu \\ \lambda & 1\end{pmatrix}, 1\right].
    \end{eqnarray*}
 Now, for $\lambda \neq 0$, we may use the decomposition
\[
        \begin{pmatrix}1 + \lambda \mu & \mu \\ \lambda & 1\end{pmatrix} = \begin{pmatrix} 1 & \lambda^{-1} +\mu  \\ 0 & 1\end{pmatrix} \begin{pmatrix} 0 & -\lambda^{-1} \\ \lambda & 1\end{pmatrix}.
    \]
Then, separating the $\lambda = 0$ part, we have 
\begin{eqnarray*}
  T_{1,2}(s_1^{k}) & = & \sum_{\mu \in I_1}\mu^{k}\left[\begin{pmatrix} 1 & \mu  \\ 0 & 1\end{pmatrix}, 1\right] + \sum_{\mu \in I_{1}} \sum_{0 \neq \lambda \in I_1}\mu^k \left[\begin{pmatrix} 1 & \lambda^{-1} +\mu  \\ 0 & 1\end{pmatrix}w\begin{pmatrix} \lambda & 1 \\ 0 &  -\lambda^{-1}\end{pmatrix}, 1\right] \\
  & = & \sum_{\mu \in I_1}\mu^{k}\left[\id , 1\right] +\sum_{\mu \in I_{1}} \sum_{0 \neq \lambda \in I_1} \mu^k \left[\begin{pmatrix} 1 & [\lambda^{-1} +\mu]  \\ 0 & 1\end{pmatrix}w, 1\right]. 
\end{eqnarray*}
In the last equality, we have used the fact that $\lambda^{-1} + \mu$ and $[\lambda^{-1} + \mu]$ differ by an element in $\varpi \mathcal{O}_F$ and that $w$ normalizes $K(1)$. Making the substitution $\lambda \mapsto (\lambda - \mu)^{-1}$, we see that
\begin{eqnarray*}  
  T_{1, 2}(s_1^k)& = & \left(\sum_{\mu \in I_1}\mu^{k}\right)\left[\id , 1\right] +\sum_{\mu \in I_{1}} \sum_{ \lambda \neq \mu \in I_1} \mu^k \left[\begin{pmatrix} 1 & \lambda  \\ 0 & 1\end{pmatrix}w, 1\right]\\
  & = & \left(\sum_{\mu \in I_1}\mu^{k}\right)\left[\id , 1\right] +\sum_{\mu \in I_{1}} \left(\sum_{\lambda \in I_1} \mu^k \left[\begin{pmatrix} 1 & \lambda  \\ 0 & 1\end{pmatrix}w, 1\right]- \mu^k \left[\begin{pmatrix} 1 & \mu  \\ 0 & 1\end{pmatrix}w , 1\right] \right)\\
   & = & \left(\sum_{\mu \in I_1}\mu^{k}\right)\left[\id , 1\right] +\sum_{\lambda \in I_{1}} \left(\sum_{\mu \in I_1} \mu^k \right)\left[\begin{pmatrix} 1 & \lambda  \\ 0 & 1\end{pmatrix}w, 1\right]-\sum_{\mu \in I_1}\mu^k \left[\begin{pmatrix} 1 & \mu  \\ 0 & 1\end{pmatrix}w , 1\right]\\
   & = & \left(\sum_{\mu \in I_1}\mu^{k}\right)\left[\id , 1\right] +\sum_{\lambda \in I_{1}} \left(\sum_{\mu \in I_1} \mu^k \right)\left[\begin{pmatrix} 1 & \lambda  \\ 0 & 1\end{pmatrix}w, 1\right]- t_1^k.
 \end{eqnarray*}
 So, we have 
\begin{equation}\label{T_1,2 applied on s family}
    T_{1,2}(s_1^k)= \left(\sum_{\mu \in I_1}\mu^{k}\right)\left[\id , 1\right] +\sum_{\lambda \in I_{1}} \left(\sum_{\mu \in I_1} \mu^k \right)\left[\begin{pmatrix} 1 & \lambda  \\ 0 & 1\end{pmatrix}w, 1\right]- t_1^k.
\end{equation}
 
Using the identity 
\begin{eqnarray}\label{sum of powers of roots of unity}
            \sum_{\zeta \in \Fq}\zeta^l = 
            \begin{cases}
                0 & \text{ if } q - 1 \nmid l \\
                q - 1 & \text{ if } q - 1 \mid l,
                \end{cases}
                & \text{ for $l \geq 1$,}
            \end{eqnarray} 
to sum over $\mu \in I_1$, we get  $t_1^{k} \in {\im}T_{1,2}$. 
Thus, the lemma has been proved for $k \neq q-1$. 

Next, assume that $k=q-1$. Note that
\[\sum_{\mu \in I_n}c_{[\mu]_{n - 1}}\mu_{n - 1}^{q-1}\left[g^0_{n - 1, [\mu]_{n - 1}}\begin{pmatrix}1 & [\mu_{n - 1}] \\ 0 & 1\end{pmatrix}w, 1\right]= \sum_{\mu \in I_n}c_{[\mu]_{n - 1}}\mu_{n - 1}^{q-1}\left[g^0_{n, \mu} \beta, 1\right].\] 
Also, notice that 
\begin{equation*}
\sum_{\mu \in I_n}c_{[\mu]_{n - 1}}\mu_{n - 1}^{q-1}\left[g^0_{n, \mu} \beta, 1\right]= \sum_{\mu \in I_n}c_{[\mu]_{n - 1}}\mu_{n - 1}^{0}\left[g^0_{n, \mu} \beta, 1\right]- \sum_{\mu \in I_n, \mu_{n-1}=0}c_{[\mu]_{n - 1}}\left[g^0_{n, \mu} \beta, 1\right].    
\end{equation*}
Using the first case of Lemma \ref{tnk like terms}, we have the following:
\begin{equation}
\sum_{\mu \in I_n}c_{[\mu]_{n - 1}}\mu_{n - 1}^{q-1}\left[g^0_{n, \mu} \beta, 1\right]= - \sum_{\mu \in I_n, \mu_{n-1}=0}c_{[\mu]_{n - 1}}\left[g^0_{n, \mu} \beta, 1\right]  
\end{equation}
modulo $({\im} T_{1, 2}, {\Ker} T_{1, 2})$. Using the isomorphism given in \cite[Theorem 3]{CS25}, we get 
\begin{eqnarray*}
   - \sum_{\mu \in I_n, \mu_{n-1}=0}c_{[\mu]_{n - 1}}\left[g^0_{n, \mu} \beta, 1\right]   & \mapsto & - \sum_{\mu \in I_n, \mu_{n-1}=0}c_{[\mu]_{n - 1}}\left[g^0_{n, \mu} \alpha, 1\right] \\
   & = &- \sum_{\mu \in I_n, \mu_{n-1}=0}c_{[\mu]_{n - 1}}\left[g^0_{n-1, [\mu]_{n-1}}, 1\right].       
\end{eqnarray*}
Again using the isomorphism given in \cite[Theorem 3]{CS25}, we have
\begin{eqnarray*}
- \sum_{[\mu]_{n-1} \in I_{n-1} }c_{[\mu]_{n - 1}}\left[g^0_{n-1, [\mu]_{n-1}}, 1\right]\mapsto - \sum_{[\mu]_{n-1} \in I_{n-1} }c_{[\mu]_{n - 1}}\left[g^0_{n-1, [\mu]_{n-1}}, 1\right].
\end{eqnarray*}
So, we have 
\begin{equation*}
    \sum_{\mu \in I_n}c_{[\mu]_{n - 1}}\mu_{n - 1}^{q-1}\left[g^0_{n, \mu} \beta, 1\right] =  - \sum_{[\mu]_{n-1} \in I_{n-1} }c_{[\mu]_{n - 1}}\left[g^0_{n-1, [\mu]_{n-1}}, 1\right] \mod ({\im}T_{1,2},{\Ker}T_{1,2}).
\end{equation*}
This proves the second case of the lemma. \qedhere
\end{proof}

The following lemma shows that, modulo $({\im} T_{1, 2}, {\Ker} T_{1, 2})$, a function supported on a vector in $B(n)$ directed towards the root vertex reduces to a function supported on a vector in $B(n-2)$ directed towards the root vertex.

\begin{lemma}\label{snk like terms}
    Let $n \geq 2$ and $0\leq k \leq q-1$. Then, modulo $({\im} T_{1, 2}, {\Ker} T_{1, 2})$ we have
    \[
        \sum_{\mu \in I_n}c_{[\mu]_{n - 2}}\mu_{n - 2}^{k}\left[g^0_{n, \mu}, 1\right] =
        \begin{cases}
            0 & \text{ if } k \neq q - 1, \\
            \displaystyle\sum_{[\mu]_{n - 2} \in I_{n - 2}}\!\!\!\!\!c_{[\mu]_{n - 2}}[g^0_{n - 2, [\mu]_{n - 2}}, 1] & \text{ if } k = q - 1,
        \end{cases}
    \]
where $c_{[\mu]_{n - 2}} := c(\mu_0, \ldots, \mu_{n - 3})$.
\end{lemma}
\begin{proof}
    Let $f = \sum_{\mu \in I_n}c_{[\mu]_{n - 2}}\mu_{n - 2}^k [g^0_{n, \mu}, 1]$. Then,
    \[
        f = T_{-1, 0}\!\!\!\!\!\!\sum_{[\mu]_{n - 1} \in I_{n - 1}}\!\!\!\!\!\!c_{[\mu]_{n - 2}}\mu_{n - 2}^k[g^0_{n - 1, [\mu]_{n - 1}}, 1] \\
        \equiv -\!\!\!\!\!\sum_{[\mu]_{n - 1} \in I_{n - 1}} \!\!\!\!\!\!c_{[\mu]_{n - 2}}\mu_{n - 2}^k\!\!\left[g^0_{n - 2, [\mu]_{n - 2}}\!\!\begin{pmatrix}1 & [\mu_{n - 2}] \\ 0 & 1\end{pmatrix}w, 1\right].
    \]
    We then conclude the proof by applying Lemma~\ref{tnk like terms}.
\end{proof}
The rest of the section is devoted to analyzing the functions $s_n^k$ and $t_n^k$ modulo $({\im} T_{1, 2}, {\Ker} T_{1, 2})$.
\begin{remark}\label{t family}
Note that applying Lemma \ref{tnk like terms} with $k \neq q-1$ and $c_{[\mu]_{n - 1}}=1$, we get $t_n^k \equiv 0 \mod {\im}T_{1,2}$ for $n\geq 1$. Also, from the observation $-t_1^{0} = T_{1,2}(s_1^{0})$ and using \eqref{T_1,2 applied on s family} for $k=q-1$, we find that $t_1^{q-1} \equiv -\left[\id, 1\right] \mod{({\im}T_{1,2}, {\Ker}{T_{1,2}})}$.
For $t_2^{q-1}$, taking $n=2$ and $c_{[\mu]_{n-1}} = 1$ in \eqref{sum in t_1 family}, and using $t_1^{q-1}=-\left[\id, 1\right]+ T_{1,2}(s)$, we can deduce that $t_2^{q-1}= -T_{-1,0}\left(\left[\id,1\right]\right)+T_{1,2}(s')$. Therefore, we have $t_2^{q-1} \equiv \left[\beta, 1\right] \mod{({\im}T_{1,2}, {\Ker}{T_{1,2}})}$ by Lemma \ref{L1}(ii).
\end{remark}

Now, we analyze $t_n^{q-1}$ for $n \geq 3$ modulo $({\im}T_{1,2}, {\Ker}{T_{1,2}})$.
\begin{lemma}
  For $n \geq 3$, we have $t_n^{q-1} \equiv 0 \mod ({\im}T_{1,2}, {\Ker}{T_{1,2}})$.  
\end{lemma}

\begin{proof}
Indeed,
\begin{eqnarray*}
 t_n^{q-1}&= &\sum\limits_{\substack{[\mu]_{n-1}\in I_{n-1}}}\left(\begin{array}{cc}
\varpi^{n-1} & [\mu]_{n-1} \\
0    &  1
\end{array}\right) t_1^{q-1} \\
&=& \sum\limits_{\substack{[\mu]_{n-1}\in I_{n-1}}}\left(\begin{array}{cc}
\varpi^{n-1} & [\mu]_{n-1} \\
0    &  1
\end{array}\right) (-\left[\id, 1\right]+ T_{1,2}(s)) \\
&\equiv & -T_{-1,0}^{n-1}\left([\id, 1]\right) \mod{({\im}T_{1,2}, {\Ker}{T_{1,2}})}\\
&\equiv & -T_{-1,0}\left(T_{-1,0}^{n-2}\left([\id, 1]\right)\right) \mod{({\im}T_{1,2}, {\Ker}{T_{1,2}})}.
\end{eqnarray*}
The second equality follows from Remark \ref{t family}. Using Lemma \ref{L1}(ii), the last expression becomes
\[T_{1,0}\left(T_{-1,0}^{n-2}\left([\id, 1]\right)\right)
\equiv  T_{1,0}T_{-1,0}\left(T_{-1,0}^{n-3}\left([\id, 1]\right)\right) \mod{({\im}T_{1,2}, {\Ker}{T_{1,2}})}.\]
Using \eqref{T-10 in the non-commutative case}, the expression on the right is $0 \mod{({\im}T_{1,2}, {\Ker}{T_{1,2}})}$. 
\end{proof}


\begin{lemma}\label{info about sn0}
 We have, 
 \begin{enumerate}
     \item[(i)] $s_1^0 \equiv -\left[\beta, 1\right] \mod ({\im}T_{1,2}, {\Ker}{T_{1,2}})$ ,
     \item[(ii)] $s_n^0 \equiv 0 \mod ({\im}T_{1,2}, {\Ker}{T_{1,2}})$ for $n \geq 2$.
 \end{enumerate}
\end{lemma}

\begin{proof}
It is easy to see that $s_1^0=T_{-1,0}\left(\left[\id,1\right]\right) \equiv -\left[\beta,1 \right] \mod{({\im}T_{1,2}, {\Ker}{T_{1,2}})}.$

For $n \geq 2$, we have
\[s_n^0 = T_{-1,0}^n\left(\left[\id,1\right]\right)
   =  T_{-1,0} \left(T_{-1,0}^{n-1}\left(\left[\id,1\right]\right)\right).\]
Using Lemma \ref{L1}(ii), the last expression reduces to
$-T_{1,0} \left(T_{-1,0}^{n-1}\left(\left[\id,1\right]\right)\right)$, which is congruent to $0 \mod{({\im}T_{1,2}, {\Ker}{T_{1,2}})}$ by \eqref{T-10 in the non-commutative case}.
\end{proof}

Next, we show that most of the functions $s_n^k$ are not $I(1)$-invariant.

\begin{lemma}\label{snk not invariant}
    Let $n \geq 1$. The function $s_n^k$ is not $I(1)$-invariant if $2 \leq k \leq q - 1$ and $k$ is not a power of $p$.
\end{lemma}
\begin{proof}
    A short computation shows that
    \begin{eqnarray*}
        \begin{pmatrix}1 & -\varpi^{n - 1} \\ 0 & 1\end{pmatrix} \cdot s_n^k - s_n^k & = & \sum_{\mu \in I_n} \left((\mu_{n - 1} + 1)^k - \mu_{n - 1}^k\right)[g^0_{n, \mu}, 1].
    \end{eqnarray*}
    This difference is in $({\im} T_{1,2} , {\Ker} T_{1,2})$ if and only if the function
    \[
        d := \sum_{\mu \in I_n}\left((\mu_{n - 1} + 1)^k - \mu_{n - 1}^k\right)[g^0_{n, \mu}, 1] \in \KZind \sigma_0
    \]
    belongs to ${\im} T$ (see \cite[Theorem 3]{CS25}). 
    Using Lemma~\ref{functions are 0}, this means $(\mu_{n - 1} + 1)^k - \mu_{n - 1}^k$ is constant with respect to $\mu_{n - 1}$, which is not true if $2 \leq k \leq q - 1$ and $k$ is not a power of $p$.
\end{proof}

\subsection{The pro-$p$-Iwahori invariants}\label{pro-p-invarints}
In this section, we analyze the functions $s_n^{p^l}$ for $n \geq 1$ and $0 \leq l \leq f - 1$. It turns out that they are $I(1)$-invariant for $n\geq 2$ and are not for $n=1$.
\begin{remark}
    Note that
    \[
        \begin{pmatrix}
            1 & -1\\
            0 & 1
        \end{pmatrix} \cdot s_1^{p^l} - s_1^{p^l} = s_1^0,
    \]
which is congruent to $[-\beta, 1]$ modulo $({\im} T_{1, 2}, {\Ker} T_{1, 2})$ by 
Lemma \ref{info about sn0}. Therefore $s_1^{p^l}$ is not $I(1)$-invariant modulo $({\im} T_{1,2} , {\Ker} T_{1,2})$ for any $0 \leq l \leq f - 1$.
\end{remark}

Next we show that $s_n^{p^l}$ is $I(1)$-invariant for $n\geq 2$ and $0\leq l \leq f-1$. We will use the identity: for any $a, b, c , d \in \mathcal{O}_F$ and $\lambda \in \Fq$,
    \begin{eqnarray}\label{Arindam's identity stolen}
        \begin{pmatrix}
            1 + \varpi a & b \\ \varpi c & 1 + \varpi d
        \end{pmatrix}
        \begin{pmatrix}
            \varpi & [\lambda] \\ 0 & 1
        \end{pmatrix}
        =
        \begin{pmatrix}
            \varpi & [\lambda + b_0] \\ 0 & 1
        \end{pmatrix}i,
    \end{eqnarray}
    where $b_0$ is the $0$-th $\varpi$-adic digit of $b$ and where $i \in I(1)$.

\begin{lemma}\label{inductive-step}
   Let $n \geq 1$. If $s_{n-1}^k$ is $I(1)$-invariant modulo $({\im} T_{1,2} , {\Ker} T_{1,2})$, then $s_n^k$ is also $I(1)$-invariant modulo $({\im} T_{1,2} , {\Ker} T_{1,2})$.
\end{lemma}
\begin{proof}
   The proof of the lemma is similar to that of \cite[Lemma 4.4]{Jan23}.
\end{proof}

Using Lemma \ref{inductive-step}, it is enough to check the invariance for $s_2^{p^l}$ under $I(1)$ for $0\leq l \leq f-1$.

\begin{proposition}\label{sn1 are invariant}
Let $e$ and $f$ denote the ramification index and residue degree of the $p$-adic number field $F$, respectively. Assume $ef > 1$ and that $q > 3$. Then, for all $g \in I(1)$, we have 
   $$g \cdot s_{2}^{p^l}-s_{2}^{p^l} \in ({\im}T_{1,2},{\Ker}T_{1,2}),$$
    for all $0 \leq l \leq f-1$. 
    
\end{proposition}
\begin{proof}
  We check the invariance under the following matrices:
$$\begin{pmatrix}
      1+\varpi a  & 0 \\
      0 & 1 
  \end{pmatrix}, \begin{pmatrix}
      1  & 0 \\
      \varpi c & 1
  \end{pmatrix}, \text{ and } \begin{pmatrix}
      1  & b \\
      0 & 1 
  \end{pmatrix},$$
  where $a,b,c \in  \mathcal{O}_F$.
 We begin by computing,
  \begin{eqnarray*}
 \begin{pmatrix}
      1+\varpi a  & 0 \\
      0 & 1 
  \end{pmatrix} \cdot s_{2}^{p^l}  & = & \sum_{\mu \in I_2}\mu_{1}^{p^l}\left[\begin{pmatrix}
      1+\varpi a  & 0 \\
      0 & 1 
  \end{pmatrix} \begin{pmatrix} \varpi^{2} & \mu  \\ 0 & 1\end{pmatrix}, 1\right] \\
  & = &  \sum_{\mu \in I_2}\mu_{1}^{p^l}\left[\begin{pmatrix}
      1+\varpi a  & 0 \\
      0 & 1 
  \end{pmatrix} \begin{pmatrix} \varpi & [\mu_{0}]  \\ 0 & 1\end{pmatrix}\begin{pmatrix} \varpi & [\mu_{1}]  \\ 0 & 1\end{pmatrix}, 1\right]\\
  & = & \sum_{\mu \in I_2}\mu_{1}^{p^l}\left[ \begin{pmatrix} \varpi & [\mu_{0}]  \\ 0 & 1\end{pmatrix}\begin{pmatrix}
      1+\varpi a  & a[\mu_{0}] \\
      0 & 1 
  \end{pmatrix}\begin{pmatrix} \varpi & [\mu_{1}]  \\ 0 & 1\end{pmatrix}, 1\right],
\end{eqnarray*}
which, using \eqref{Arindam's identity stolen}:
$$\begin{pmatrix}
      1+\varpi a  & a[\mu_{0}] \\
      0 & 1 
  \end{pmatrix}\begin{pmatrix} \varpi & [\mu_{1}]  \\ 0 & 1\end{pmatrix}= \begin{pmatrix}
      \varpi & [\mu_{1}+a_{0}\mu_{0}]\\
      0 & 1 
  \end{pmatrix}h,$$
  for some $h \in I(1)$ and making the substitution $\mu_{1}-a_{0}\mu_{0}$, equals 
  \begin{eqnarray*}
        s_{2}^{p^l}+\sum_{\mu \in I_2}(-a_{0}\mu_{0})^{p^l}\left[\begin{pmatrix} \varpi^{2} & \mu  \\ 0 & 1\end{pmatrix}, 1\right]. 
  \end{eqnarray*}
  Applying Lemma \ref{snk like terms}, we get
  $\begin{pmatrix}
      1+\varpi a  & 0 \\
      0 & 1 
  \end{pmatrix} \cdot s_{2}^{p^l} - s_{2}^{p^l} \in ({\im} T_{1,2} , {\Ker} T_{1,2})$ if $q > 2$. If $q = 2$, then this difference is equal to $[\id, 1]$, thereby showing that $s_2^{p^l}$ is not $I(1)$-invariant modulo $({\im} T_{1, 2}, {\Ker} T_{1, 2})$.\\
   We now verify invariance under the lower triangular unipotent elements in $I(1)$:
\begin{eqnarray*}
 \begin{pmatrix}
      1 & 0 \\
      \varpi c & 1 
  \end{pmatrix} \cdot s_{2}^{p^l} & = & \sum_{\mu \in I_2}\mu_{1}^{p^l}\left[\begin{pmatrix}
      1 & 0 \\
      \varpi c & 1 
  \end{pmatrix} \begin{pmatrix} \varpi^{2} & [\mu_{0}]+[\mu_1]\varpi  \\ 0 & 1\end{pmatrix}, 1\right]\\
  &=& \sum_{\mu \in I_2}\mu_{1}^{p^l}\left[\begin{pmatrix}
      1 & 0 \\
      \varpi c & 1 
  \end{pmatrix} \begin{pmatrix} \varpi & [\mu_{0}]  \\ 0 & 1\end{pmatrix}\begin{pmatrix} \varpi & [\mu_1]  \\ 0 & 1\end{pmatrix}, 1\right]\\
  &=& \sum_{\mu \in I_2}\mu_{1}^{p^l}\left[ \begin{pmatrix} \varpi & [\mu_{0}]  \\ 0 & 1\end{pmatrix} \begin{pmatrix}
      1-\varpi c[\mu_{0}] & -[\mu_{0}^2]c \\
      \varpi^2 c & 1+\varpi c[\mu_{0}] 
  \end{pmatrix}\begin{pmatrix} \varpi & [\mu_1]  \\ 0 & 1\end{pmatrix}, 1\right]\\
  &=& \sum_{\mu \in I_2}\mu_{1}^{p^l}\left[ \begin{pmatrix} \varpi & [\mu_{0}]  \\ 0 & 1\end{pmatrix} \begin{pmatrix}
      \varpi & [\mu_{1}-\mu_{0}^2c_{0}] \\
      0 & 1 
  \end{pmatrix}i, 1\right],
\end{eqnarray*}
for some $i \in I(1)$, by \eqref{Arindam's identity stolen}. Making the change of variable $\mu_{1} \mapsto \mu_{1}+c_{0}\mu_{0}^2$, we obtain
\begin{eqnarray*}
  \begin{pmatrix}
      1 & 0 \\
      \varpi c & 1 
  \end{pmatrix} \cdot s_{2}^{p^l}&=&\sum_{\mu \in I_2}(\mu_{1}+c_{0}\mu_{0}^2)^{p^l}\left[\begin{pmatrix} \varpi^{2} & \mu  \\ 0 & 1\end{pmatrix}, 1 \right]\\ 
  &=& s_{2}^{p^l}+\sum_{\mu \in I_2}(c_{0}\mu_{0}^2)^{p^l}\left[\begin{pmatrix} \varpi^{2} & \mu  \\ 0 & 1\end{pmatrix}, 1 \right].
\end{eqnarray*}
Applying Lemma \ref{snk like terms} once more, we conclude (for $q > 3$) that
\[
\begin{pmatrix}
      1 & 0 \\
      \varpi c & 1 
  \end{pmatrix} \cdot s_{2}^{p^l}-s_{2}^{p^l} \in ({\im} T_{1,2} , {\Ker} T_{1,2}) .
\]
Now, we compute
\begin{eqnarray*}
 \begin{pmatrix}
      1 & b \\
      0 & 1 
  \end{pmatrix} \cdot s_{2}^{p^l} & = & \sum_{\mu \in I_2}\mu_{1}^{p^l}\left[\begin{pmatrix}
      1 & b \\
      0 & 1 
  \end{pmatrix} \begin{pmatrix} \varpi^{2} & [\mu_{0}]+[\mu_1]\varpi  \\ 0 & 1\end{pmatrix}, 1\right]\\
  &=& \sum_{\mu \in I_2}\mu_{1}^{p^l}\left[\begin{pmatrix} \varpi^{2} & [\mu_{0}]+[\mu_1]\varpi+b  \\ 0 & 1\end{pmatrix}, 1\right]\\
  &=& \sum_{\mu \in I_2}\mu_{1}^{p^l}\left[\begin{pmatrix} \varpi & [\mu_{0}+b_0]  \\ 0 & 1\end{pmatrix}\begin{pmatrix} \varpi & [\mu_1+b_1+Z]  \\ 0 & 1\end{pmatrix}i, 1\right],
\end{eqnarray*}
for some $i \in I(1)$, where the last equality follows from \eqref{Arindam's identity stolen}. Here
\[
Z=\begin{cases}
      0, &  \text{if } e>1, \\
      \dfrac{\mu_0^{q}+b_0^{q}-(\mu_{0}+b_0)^q}{p}, & \text{if } e=1.
  \end{cases}
\] \\
Next, we perform the change of variables
$$\mu_{0} \rightarrow \mu_{0}-b_0 \quad  \text{and} \quad \mu_{1} \rightarrow \mu_{1}-b_1-Z,$$
which yields
\begin{eqnarray*}
 \sum_{\mu \in I_2}(\mu_{1}-b_1-Z)^{p^l}\left[\begin{pmatrix} \varpi^2 & \mu  \\ 0 & 1\end{pmatrix}, 1\right]
 &= & s_{2}^{p^l}+\sum_{\mu \in I_2}(-b_1-Z)^{p^l}\left[\begin{pmatrix} \varpi^2 & \mu  \\ 0 & 1\end{pmatrix}, 1\right].
\end{eqnarray*}
If $e > 1$, then $Z = 0$ and therefore the largest power of $\mu_0$ appearing in the coefficient is at most $0$. If $e = 1$, then $f > 1$ and therefore the largest power of $\mu_0$ appearing in the coefficient is $(p - 1)p^{f - 1} < q - 1$. In both cases, Lemma~\ref{snk like terms} applies, and we get
\[
\begin{pmatrix}
      1 & b \\
      0 & 1
\end{pmatrix} \cdot s_{2}^{p^l}-s_2^{p^l} \in ({\im} T_{1,2}, {\Ker} T_{1,2}).
\]
This completes the proof of the proposition.
\end{proof}

Next, we show that the functions $s^{p^l}_n$ for $n \geq 2$ and $0 \leq l \leq f - 1$ are new $I(1)$-invariants in addition to the obvious ones, namely $[\id, 1]$ and $[\beta, 1]$.

\begin{lemma}\label{eigenvalues of sn}
    The $I$-eigencharacter of $s_n^{p^l}$ for $n \geq 2$ and $0 \leq l \leq f - 1$ is $(d/a)^{p^l}$.
\end{lemma}
\begin{proof}
    Let $t = \begin{pmatrix}a & 0 \\ 0 & d\end{pmatrix} \in I$. Note that
    \[
        t g^0_{n, \lambda} = g^0_{n, a\lambda/d}t.
    \]
    Making the change of coordinates $\mu_{m} \mapsto d\mu_{m}/a$, we see that
    \[
        t \cdot s_n^{p^l} = (d/a)^{p^l}s_n^{p^l}.\qedhere
    \]
\end{proof}

\begin{proposition}\label{linear independence lemma}
    The set $\left\{[\id, 1], [\beta, 1], s_n^{p^l}, \beta s_n^{p^l}\right\}_{\substack{n \geq 2 \\ 0 \leq l \leq f - 1}}$ is linearly independent.
\end{proposition}
\begin{proof}
    Using the fact that $I$ acts trivially on $[\id, 1]$ and $[\beta, 1]$ with Lemma~\ref{eigenvalues of sn}, we see that if there is a dependence relation, then it must consist only of $s_n^{p^l}$ or only of $\beta s_n^{p^l}$ for a fixed $0 \leq l \leq f - 1$. This proposition is then proved by using the isomorphism
    \begin{eqnarray*}
        \frac{\IZind \mathbbm{1}}{({\im} T_{1,2} , {\Ker} T_{1,2})} & \xrightarrow{\sim} & \frac{\KZind \sigma_0}{{\im} T} \\
        \left[\id, 1\right] & \mapsto & [\id, 1],
    \end{eqnarray*}
    proved in \cite[Theorem 3]{CS25}, with Lemma~\ref{functions are 0}.
\end{proof}

We remark that Proposition \ref{sn1 are invariant} excludes $q = 2, 3$. These were also excluded in {\rm{\cite[Theorem $1.2$]{Hen19}}}, {\rm{\cite[Theorem $1.4$]{Jan23}}} and {\rm{\cite[Theorem $1.2$]{AJ21}}}. Indeed, for $f=1$ they assumed that the $p$-adic digits of $r$ lie in $(2, p - 3)$. This means that their results do not apply to $q = 2, 3$ and $5$. We end this section by addressing the case $q=3$ in the following remark.

\begin{remark}\label{q = 2, 3}
    Note that for $q = 2, 3$, the element $s_2^{1}$ is not invariant under the lower triangular unipotents in $I(1)$ thanks to Lemma~\ref{snk like terms}. We show that for $q = 3$, the function $s_3^{1}$ is $I(1)$-invariant. Using Lemma \ref{inductive-step}, we then see that 
 $s_n^{1}$ is $I(1)$-invariant for $n \geq 3$. The checks for upper triangular unipotents and the diagonals in $I(1)$ are similar to those done in Proposition~\ref{sn1 are invariant}. We now provide an argument for the lower triangular unipotents. Consider,
 \begin{eqnarray*}
 \begin{pmatrix}
      1 & 0 \\
      \varpi c & 1 
  \end{pmatrix} \cdot s_{3}^{1} & = & \sum_{\mu \in I_3}\mu_{2}\left[\begin{pmatrix}
      1 & 0 \\
      \varpi c & 1 
  \end{pmatrix} \begin{pmatrix} \varpi^{3} & [\mu_{0}]+[\mu_1]\varpi+[\mu_2]\varpi^2  \\ 0 & 1\end{pmatrix}, 1\right]\\
  &=& \sum_{\mu \in I_3}\mu_{2}\left[\begin{pmatrix}
      1 & 0 \\
      \varpi c & 1 
  \end{pmatrix} \begin{pmatrix} \varpi & [\mu_{0}]  \\ 0 & 1\end{pmatrix}\begin{pmatrix} \varpi & [\mu_1]  \\ 0 & 1\end{pmatrix}\begin{pmatrix}
      \varpi & [\mu_2] \\
      0   & 1 
  \end{pmatrix}, 1\right]\\
  &=& \sum_{\mu \in I_3}\mu_{2}\left[ \begin{pmatrix} \varpi & [\mu_{0}]  \\ 0 & 1\end{pmatrix} \begin{pmatrix}
      1-\varpi c[\mu_{0}] & -[\mu_{0}^2]c \\
      \varpi^2 c & 1+\varpi c[\mu_{0}] 
  \end{pmatrix}\begin{pmatrix} \varpi & [\mu_1]  \\ 0 & 1\end{pmatrix} \begin{pmatrix}
      \varpi & [\mu_2] \\
      0   & 1 
  \end{pmatrix}, 1\right]\\
  &=& \sum_{\mu \in I_3}\mu_{2}\left[ \begin{pmatrix} \varpi & [\mu_{0}]  \\ 0 & 1\end{pmatrix} \begin{pmatrix}
     \varpi & [\mu_1-\mu_0^2c_0] \\
      0   & 1 
  \end{pmatrix}i \begin{pmatrix}
      \varpi & [\mu_2] \\
      0   & 1 
  \end{pmatrix}, 1\right],
\end{eqnarray*}
  where $$i = \begin{pmatrix}1 + \varpi\Delta & [c_0^2\mu_0^3-2c_0\mu_0\mu_1 - \mu_0^2c_1] + \varpi\Delta' \\ \varpi\Delta'' & 1 + \varpi\Delta''' \end{pmatrix}$$
  for some $\Delta, \Delta', \Delta'',$ and $\Delta''' \in \mathcal{O}$. Here we have used {\rm\cite[Lemma 4.1]{Jan23}} and the fact that $e > 1$ since $q = 3$. So, we have 
\begin{eqnarray*}
 \begin{pmatrix}
      1 & 0 \\
      \varpi c & 1 
  \end{pmatrix} \cdot s_{3}^{1} & = & \sum_{\mu \in I_3}\mu_{2}\left[ \begin{pmatrix} \varpi & [\mu_{0}]  \\ 0 & 1\end{pmatrix} \begin{pmatrix}
     \varpi & [\mu_1-\mu_0^2c_0] \\
      0   & 1 
  \end{pmatrix} \begin{pmatrix}
      \varpi & [\mu_2 + c_0^2\mu_0^3 - 2c_0\mu_0\mu_1 -\mu_0^2c_1] \\
      0   & 1 
  \end{pmatrix}, 1\right].
  \end{eqnarray*}
Making the change of variables $\mu_1 \mapsto \mu_1+\mu_0^2c_0 \text{ and } \mu_2 \mapsto \mu_2 + c_0^2\mu_0^3 + 2c_0\mu_0\mu_1 + \mu_0^2c_1,$ gives 
\begin{eqnarray*}
 \begin{pmatrix}
      1 & 0 \\
      \varpi c & 1 
  \end{pmatrix} \cdot s_{3}^{1} & = & \sum_{\mu \in I_3}(\mu_{2}+ c_0^2\mu_0^3 + 2c_0\mu_0\mu_1 + \mu_0^2c_1)\left[ \begin{pmatrix} \varpi & \mu  \\ 0 & 1\end{pmatrix} , 1\right].
  \end{eqnarray*}
Applying Lemma \ref{snk like terms}, we get (for $q=3$),
\[
\begin{pmatrix}
      1 & 0 \\
      \varpi c & 1
\end{pmatrix} \cdot s_{3}^{1}-s_3^{1} \in ({\im} T_{1,2}, {\Ker} T_{1,2}).
\]
\end{remark}

\section{Proof of Theorems \ref{pro-p-iwahori-invariants-intro} and \ref{endo-algebra-intro}}\label{proof-of- theorems}
In this section we first prove Theorem \ref{pro-p-iwahori-invariants-intro}. Then using the explicit description of $I(1)$-invariants obtained in Theorem \ref{pro-p-iwahori-invariants-intro}, we give a proof of Theorem \ref{endo-algebra-intro}. We start off by proving some essential lemmas.

The coefficients in a linear combination of functions in the induced space are given by polynomials (see Lemma \ref{Every fun is a poly}). In particular, if a function $f$ is a linear combination of functions in $B(n)\setminus B(n-1)$, then 
\[f = \sum_{\mu \in I_n}a_{\mu}\left[g^0_{n, \mu}, 1\right] +  \sum_{\mu \in I_n}b_{\mu}\left[g^0_{n, \mu}\beta, 1\right] + \sum_{\mu \in I_{n-1}}c_{\mu}\left[g^1_{n-1, \mu} w, 1\right] +  \sum_{\mu \in I_{n-1}}d_{\mu}\left[g_{n-1,\mu}^1w \beta, 1\right],\]
where $a_\mu, b_\mu$ are polynomials in $\mu_0, \dots, \mu_{n-1}$ and $c_\mu$ and $d_\mu$
are polynomials in $\mu_0, \dots, \mu_{n-2}$. The next lemma predicts the powers of $\mu_{n-1}$ for $f$ to be invariant under $\small\begin{pmatrix}1 & -\varpi^{n - 1} \\ 0 & 1\end{pmatrix} \in I(1)$.
\begin{lemma}\label{A form of invariants}
    For $n \geq 1$, let $f_n \in \IZind \mathbbm{1}$ be a function of the form $f_n' + f_n''$, where
    \[
        f_n' = \sum_{\mu \in I_n}a_{\mu}\left[g^0_{n, \mu}, 1\right] \text{ and } f_n'' = \sum_{\mu \in I_n}b_{\mu}\left[g^0_{n, \mu}\beta, 1\right], 
    \]
    for some polynomials $a_\mu$ and $b_\mu$ in variables $\mu_0, \mu_1, \ldots, \mu_{n - 1}$. We may assume that modulo $({\im}T_{1, 2}, {\Ker} T_{1, 2})$, the power of $\mu_{n - 1}$ in $b_{\mu}$ is $q - 1$. If $f_n$ is invariant under $\begin{pmatrix}1 & -\varpi^{n - 1} \\ 0 & 1\end{pmatrix}$ modulo $({\im} T_{1,2} , {\Ker} T_{1,2})$, then the powers of $\mu_{n - 1}$ in $a_\mu$ are $0$ or $p^l$ for $0 \leq l \leq f - 1$. 
\end{lemma}
\begin{proof}
    The $b_{\mu}$-part of this lemma follows from Lemma~\ref{tnk like terms}.
    We prove the $a_\mu$-part of this lemma. Note that
    \[
        \begin{pmatrix}1 & -\varpi^{n - 1} \\ 0 & 1\end{pmatrix}\cdot f_n - f_n= \left(\begin{pmatrix}1 & -\varpi^{n - 1} \\ 0 & 1\end{pmatrix}\cdot f_n' - f_n'\right) + \left(\begin{pmatrix}1 & -\varpi^{n - 1} \\ 0 & 1\end{pmatrix} \cdot f_n'' - f_n'' \right).
    \]
    The expression in the second parentheses above
    \[
        \begin{pmatrix}1 & -\varpi^{n - 1} \\ 0 & 1\end{pmatrix}\cdot f_n'' - f_n'' = \sum_{\mu \in I_n} \sum_{i = 0}^{q - 1} c_{[\mu]_{n - 1}, i}[(\mu_{n - 1} + 1)^i - \mu_{n - 1}^i]\left[g^0_{n - 1, [\mu]_{n - 1}}\begin{pmatrix}1 & [\mu_{n - 1}] \\ 0 & 1\end{pmatrix}w, 1\right]
    \]
    belongs to ${\im} T_{1, 2}$ by Lemma~\ref{tnk like terms}. Therefore, $f_n'$ and $f_n''$ are separately invariant under the action of $\begin{pmatrix}1 & -\varpi^{n - 1} \\ 0 & 1\end{pmatrix}$ modulo $({\im} T_{1, 2}, {\Ker} T_{1, 2})$.

    We write
    \[
        a_\mu = \sum_{i = 0}^{q - 1}a_{[\mu]_{n - 1},i}' \mu_{n - 1}^i,
    \]
    where the $a_{[\mu]_{n - 1}, i}'$ are functions of $\mu_{0}, \ldots, \mu_{n - 2}$. Note that
    \begin{eqnarray*}
        \begin{pmatrix}1 & -\varpi^{n - 1} \\ 0 & 1\end{pmatrix} \cdot f_n' - f_n' & = & \sum_{\mu \in I_n}\sum_{i = 1}^{q - 1}a_{[\mu]_{n - 1},i}'\left((\mu_{n - 1} + 1)^i - \mu_{n - 1}^i\right)\left[g^0_{n, \mu}, 1\right].
    \end{eqnarray*}
        Using the invariance of $f_n'$ proved just above, this difference belongs to $({\im} T_{1,2} , {\Ker} T_{1,2})$. Then, using the isomorphism  $\frac{\IZind \mathbbm{1}}{({\im} T_{1,2} , {\Ker} T_{1,2})} \xrightarrow{\sim} \frac{\KZind \sigma_0}{{\im} T}$ given in \cite[Theorem 3]{CS25} together with Lemma~\ref{functions are 0} we see that the function $\mu_{n - 1} \mapsto \sum_{i = 1}^{q - 1}a_{[\mu]_{n - 1},i}'\left((\mu_{n - 1} + 1)^i - \mu_{n - 1}^i\right)$ is constant. Since this is a polynomial of degree less than or equal to $q - 1$, we see that $a_{[\mu]_{n - 1},i}' = 0$ for all $2 \leq i \leq q - 1$ not a power of $p$. This concludes the proof of the $a_{\mu}$-part of the lemma.
\end{proof}

\begin{lemma}\label{independence wrt smaller mus}
    For $n \geq 1$, let $f_n \in \IZind \mathbbm{1}$ be a function of the form $f_n^0 + f_n^1$, where
    \begin{eqnarray*}
        && f_n^0 = \sum_{\mu \in I_n}a_{\mu}\left[g^0_{n, \mu}, 1\right] + \sum_{\mu \in I_n}b_{\mu}\left[g^0_{n, \mu}\beta, 1\right], \\ 
        &&f_n^1 = \sum_{\mu \in I_{n - 1}}c_{\mu}[g_{n-1,\mu}^1 w,1] + \sum_{\mu \in I_{n - 1}} d_{\mu}\left[g_{n-2,[\mu]_{n-2}}^1w \left(\begin{array}{cc}1& [\mu_{n-2}] \\ 0 & 1\end{array}\right)w,1\right],
    \end{eqnarray*}
    for some polynomials $a_\mu$, $b_\mu$, $c_{\mu}$ and $d_{\mu}$ in variables $\mu_0, \mu_1, \ldots, \mu_{n - 1}$. We also assume that the powers of $\mu_{n - 1}$ in $a_{\mu}$ are non-zero. Let $f = f_n + f'$, where $f' \in B(n - 1)$. Then,
    \begin{enumerate}
        \item if $f$ is invariant under $\begin{pmatrix}1 & -\varpi^{n - m} \\ 0 & 1\end{pmatrix}$ modulo $({\im} T_{1, 2}, {\Ker} T_{1, 2})$ for $1 \leq m \leq n$, then $a_\mu$ 
    is constant with respect to $\mu_0, \cdots, \mu_{n - 2}$.
    \item if $f$ is invariant under $\begin{pmatrix}1 & 0 \\ -\varpi^{n - m +1} & 1\end{pmatrix}$ modulo $({\im} T_{1, 2}, {\Ker} T_{1, 2})$ for $3 \leq m \leq n$, then $c_\mu$ is constant with respect to $\mu_0, \cdots, \mu_{n - 3}$. 
    \end{enumerate}
    \end{lemma}
\begin{proof}
 We shall prove the lemma for $a_\mu$. A similar proof goes through for $c_\mu$. 
    
    Assume that for some $0 \leq j \leq n - 2$, we can write 
    \[
        a_\mu = \sum_{l = 0}^{f - 1}\sum_{i = 0}^{q - 1}a_{[\mu]_{j}, i, l}\mu_{j}^i \mu_{n - 1}^{p^l} 
    \]
    where $a_{[\mu]_j, i, l}$ 
    is a function depending only on $[\mu]_j$. Here we have used Lemma~\ref{A form of invariants} (with the fact that the matrix $\begin{pmatrix}1 & -\varpi^{n - 1} \\ 0 & 1\end{pmatrix}$ fixes $f_n^1$ and $f'$) to restrict the powers of $\mu_{n - 1}$ to $p^{l}$, $0 \leq l \leq f - 1$ 
    in $a_{\mu}$.
    We write 
    \begin{eqnarray}\label{gf-f}
        \begin{pmatrix}
            1 & -\varpi^j \\ 0 & 1
        \end{pmatrix} \cdot f - f = \begin{pmatrix}
            1 & -\varpi^j \\ 0 & 1
        \end{pmatrix} \cdot f_n^0 - f_n^0 + f'',
    \end{eqnarray}
    where the support of $f''$ consists of edges in $B(n - 1)$ and the edges in $B(n)$ on the right-hand side of the tree with respect to the orientation chosen in figure \ref{Bruhat-Tits tree}.
    A short computation shows
    \begin{align}\label{The difference ifn - fn}
        \begin{pmatrix}
            1 & -\varpi^j \\ 0 & 1
        \end{pmatrix} \cdot f_n^0 - f_n^0 = & \sum_{\mu \in I_n}\sum_{l = 0}^{f - 1}\sum_{i = 0}^{q - 1}a_{[\mu]_{j}, i, l}\left((\mu_{j} + 1)^i - \mu_{j}^i\right)\mu_{n - 1}^{p^l}[g^0_{n, \mu}, 1] \nonumber\\
        & + \sum_{\mu \in I_n}\sum_{l = 0}^{f - 1}\sum_{i = 0}^{q - 1}a_{[\mu]_{j}, i, l}(\mu_j + 1)^iz_{n - 1}^{p^l}[g^0_{n, \mu}, 1] \\
        & + \sum_{\mu \in I_n}\sum_{i = 0}^{q - 1}b_{[\mu]_j, i}\left((\mu_{j} + 1)^i(\mu_{n - 1} + z_{n - 1})^{q - 1} - \mu_{j}^i\mu_{n - 1}^{q - 1}\right)[g^0_{n, \mu}\beta, 1], \nonumber
    \end{align}
    for some $z_{n - 1} \in \Fq$ depending on $\mu_{j}, \ldots, \mu_{n - 2}$.
    The hypothesis on $f$ says that the left-hand side of the equation \eqref{gf-f} belongs to $({\im}T_{1, 2}, {\Ker} T_{1, 2})$. So the right-hand side does too. Using the isomorphism
    \[
        \frac{\IZind \mathbbm{1}}{({\im} T_{1, 2}, {\Ker} T_{1, 2})} \xrightarrow{\sim}\frac{\KZind \sigma_0}{{\im} T},
    \]
     given in \cite[Theorem 3]{CS25} together with \eqref{gf-f}, we see that the images of the first two sums on the right-hand side of \eqref{The difference ifn - fn} are supported on the vertices at a distance of $n$ from the root vertex, whereas the image of the last sum is supported on the vertices at a distance of $n - 1$ from the root vertex. Then, Lemma~\ref{functions are 0} applies and we conclude that the function
    \[
        \mu_{n - 1} \mapsto \sum_{l = 0}^{f - 1}\sum_{i = 0}^{q - 1}a_{[\mu]_j, i, l}\left((\mu_{j} + 1)^i - \mu_{j}^i\right)\mu_{n - 1}^{p^l} + a_{[\mu]_j, i, l}(\mu_j + 1)^iz_{n - 1}^{p^l}
    \]
    is constant for every $[\mu]_{n - 1}$. Since $z_{n - 1}$ does not depend on $\mu_{n - 1}$, we may ignore the second term in the expression above. Since $p^{f - 1} < q$, the constancy of the function above with respect to $\mu_{n - 1}$ forces $\sum_{i = 0}^{q - 1}a_{[\mu]_{j}, i, l}\left((\mu_{j} + 1)^i - \mu_{j}^i\right) = 0$ for all $l = 0, \ldots, f - 1$. Fixing $[\mu]_{j}$ and varying $\mu_{j}$ through $\Fq$, results in $a_{[\mu]_{j}, i, l} = 0$ for every $0 \leq l \leq f - 1$ and every $0 \leq i \leq q - 1$, $p \nmid i$. So we may write the remaining $i$ as $pi'$ for $0 \leq i' \leq q/p -1 $ and that $\sum_{i' = 0}^{q/p-1}a_{[\mu]_j, pi', l}\left((\mu_{j}^{p} + 1)^{i'} - (\mu_{j}^p)^{i'}\right) = 0$. Since the absolute Frobenius is a bijection on $\Fq$, we may take $p$-th roots and repeat the previous argument to conclude that $a_{[\mu]_j, i, l} = 0$ for $0 \leq l \leq f - 1$ and $0 \leq i \leq q - 1$, $p^2 \nmid i$. Iterating this $f$-times, we see that the only $a_{[\mu]_j, i, l}$ that may survive are those with $i = 0$. This means that $a_{\mu}$ is constant with respect to $\mu_j$. Finally, 
    the lemma is proved by varying $j$ from $n - 2$ to $0$.
\end{proof}

Let us now prove Theorem \ref{pro-p-iwahori-invariants-intro}.
\begin{theorem}\label{pro-p-iwahori-invariants}
     Let 
    \[ 
        \tau := \frac{\IZind \mathbbm{1}}{({\im} T_{1,2} , {\Ker} T_{1,2})}
    \] 
    be the Iwahori-Hecke model of the universal supersingular representation $\pi_0$ of $G$. A basis for the set of $I(1)$-invariants of $\tau$ is
    \[
        \left\{[\id, 1], [\beta, 1], s_n^{p^l}, \beta s_n^{p^l}\right\}_{\substack{n \geq 2 \\ 0 \leq l \leq f - 1}}.
    \]
\end{theorem}
\begin{proof}
    By Proposition~\ref{sn1 are invariant} and Lemma \ref{inductive-step}, we see that the set $\left\{[\id, 1], [\beta, 1], s_n^{p^l}, \beta s_n^{p^l}\right\}_{\substack{n \geq 2 \\ 0 \leq l \leq f - 1}}$ is contained in $\tau^{I(1)}$. By Proposition~\ref{linear independence lemma}, this set is linearly independent. We show that it generates $\tau^{I(1)}$.

   Let $f \in \tau^{I(1)}$. We write 
   \[
        f = f_n + f',
   \]
   where $f_n \in B(n) \setminus B(n - 1)$ and $f' \in B(n - 1)$. We write
   \[
        f_n = f_n^0 + f_n^1,
   \]
   where $f_n^0 = \sum_{\mu \in I_n}a_{\mu}[g^0_{n, \mu}, 1] + \sum_{\mu \in I_n}b_{\mu}[g^0_{n, \mu} \beta, 1]$ for some functions $a_{\mu}, b_{\mu} : I_n \to \Fq$ is the part of $f_n$ on the left-hand side of the tree and $f_n^1$ the part of $f$ on the right-hand side of the tree. Using Lemma~\ref{Every fun is a poly} we may assume that $a_{\mu}$ and $b_{\mu}$ are polynomials in $\mu_0, \ldots, \mu_{n - 1}$. Using Lemma \ref{A form of invariants} (with the fact that the matrix $\begin{pmatrix}1 & -\varpi^{n - 1} \\ 0 & 1\end{pmatrix}$ fixes $f_n^1$ and $f'$), Lemma~\ref{independence wrt smaller mus} and Lemma \ref{snk like terms}, we see that $f \equiv \sum_{l = 0}^{f - 1}a_{n, l}s_n^{p^l} + f_n^1 + f'' \mod ({\im} T_{1, 2}, {\Ker} T_{1, 2})$ for some $a_{n, l} \in \Fq$ and for $f'' \in B(n - 1)$. Since $s_n^{p^l}$ is $I(1)$-invariant modulo $({\im} T_{1, 2}, {\Ker} T_{1, 2})$, we see that $f_n^1 + f''$ is also $I(1)$-invariant. As the matrix $\begin{pmatrix}1 & 0 \\ -\varpi^{n - 1} & 1\end{pmatrix}$ fixes $f''$, we see that it fixes $f_n^1$ modulo $({\im} T_{1,2} , {\Ker} T_{1,2})$. Note that $f_n^1$ is of the form $\beta f_{n-1}^0$ for some function $f_{n-1}^0$ supported on $B^0(n-1)\setminus B^0(n-2)$. So $\beta^{-1} \begin{pmatrix}1 & 0 \\ -\varpi^{n - 1} & 1\end{pmatrix} \beta =\begin{pmatrix}1 & -\varpi^{n - 2}\\ 0 & 1\end{pmatrix}$ fixes $\beta^{-1} f_n^1 = f_{n-1}^0$ and therefore by Lemma \ref{A form of invariants} the powers of $\mu_{n-2}$ in $f_{n-1}^0$ are either $0$ or $p^l$ for some $0\leq l \leq f-1$. Also Lemma \ref{independence wrt smaller mus} implies $f_n^1$ is congruent to $\sum_{l = 0}^{f - 1}a_{n, l}'\beta s_{n-1}^{p^l}$ modulo $({\im} T_{1, 2}, {\Ker} T_{1, 2})$. Finally we have
   \[f \equiv \sum_{l = 0}^{f - 1}a_{n, l}s_n^{p^l} + \sum_{l = 0}^{f - 1}a_{n, l}'\beta s_{n-1}^{p^l} +  f''' \mod ({\im} T_{1, 2}, {\Ker} T_{1, 2}),\] where $a_{n, l}' \in \Fq$ and for $f''' \in B(n - 1)$. Since $f, s_n^{p^l}$ and $\beta s_{n-1}^{p^l}$ are $I(1)$-invariant, so is $f'''$. The theorem is then proved by induction on $n$.
\end{proof}

    As an application of Theorem \ref{pro-p-iwahori-invariants}, we can show that the endomorphism algebra of $\pi_0$ only consists of scalars. This shows, in particular, that $\pi_0$ is indecomposable. Note that the following theorem appears as Theorem \ref{endo-algebra-intro} in the introduction.
    \begin{theorem}\label{endo-algebra}
        The endomorphism algebra $\End_G(\pi_0) = \overline{\F}_p$.
    \end{theorem}
    \begin{proof}
        Recall that $\tau$ is the Iwahori model of the universal supersingular representation of $\pi_0$. Let $f : \tau \to \tau$ be a $G$-endomorphism. Then, $f$ must take the $I(1)$-invariant vector $[\id, 1]$ to an $I(1)$-invariant vector with the same $I$-eigencharacter. This means that
        \[
            f([\id, 1]) = a[\id, 1] + b[\beta, 1].
        \]
        Applying the operator $\sum_{\lambda \in I_1}\begin{pmatrix}1 & 0 \\ \lambda & 1\end{pmatrix}$ to the equation above and using the fact that
        \[
        T_{1, 2}([\alpha^{-1}, 1]) = \sum_{\lambda \in I_1}\left[\begin{pmatrix}1 & 0 \\ \lambda & 1 \end{pmatrix}, 1\right],
        \] 
        we get
        \[0 =  b\sum_{\lambda \in I_1}\left[\begin{pmatrix}1 & 0 \\ \lambda & 1\end{pmatrix}\beta, 1\right] = b\sum_{\lambda \in I_1}\left[wg^0_{1, \lambda}, 1\right],\]
        which, (using the definition of $T_{-1,0}$) equals \[bT_{-1, 0}[w, 1] \equiv -b[w\beta, 1] \mod ({\im} T_{1, 2}, {\Ker} T_{1, 2}),\] 
        where the congruence is obtained by Lemma \ref{L1} (ii). Therefore we have
        \[-b[w\beta, 1] \equiv 0 \mod ({\im} T_{1, 2}, {\Ker} T_{1, 2}).\]
        
        Since $[w\beta, 1]$ does not belong to $({\im} T_{1, 2}, {\Ker} T_{1, 2})$, we see that $b = 0$. This means that $f$ is just multiplication by $a$.
    \end{proof}
\section{Recovering the universal module functorially}

Let ${\rm Rep}_G$ and ${\rm Mod}_\cH$ denote the category of smooth $\mathbb{\overline{F}}_{p}$-representations of $G$ and the category of right $\cH$-modules respectively. Given $\tau \in {\rm Rep}_G$, using Frobenius reciprocity for the subgroup $I(1)Z$ of $G$, the set of $I(1)$-invariants in $\tau$ admits an action of the Hecke algebra $\cH$ and hence, $\tau^{I(1)}\in {\rm Mod}_\cH$. The reverse functor studied by Vign\`eras, Ollivier, and others is given by $M \mapsto M \otimes_{\cH} \IIZind \mathbbm{1}$. In this section, we first analyze the action of $\cH$ on the $I(1)$-invariants of the universal supersingular representation $\pi_0$ 
explicitly described in Theorem \ref{pro-p-iwahori-invariants}. By proving a comparison result between the Iwahori and pro-$p$-Iwahori induction, we finally show that 
$\pi_0$ can be realized as the direct summand of the image of the reverse functor (see Theorem \ref{summand-in-reverse-vigneras-functor}).

\subsection{Action of $\cH$ on the $I(1)$-invariants}
 \label{action-on-I(1)-invariants}
    We recall the formulas for the operators $T_{\beta}$, $e_{\chi}$, and $T_{n_s}$ on $\pi_0^{I(1)}$ for all characters $\chi : H \to \Fq^{\times}$. For any $v \in \pi_0^{I(1)}$,
    \begin{enumerate}
        \item $v \cdot T_{\beta} = \beta^{-1} \cdot v$,
        \item $v \cdot e_{\chi} = |H|^{-1}\sum_{h \in H}\chi(h)h^{-1} \cdot v$, and
        \item $v \cdot T_{n_s} = \sum_{\lambda \in I_1}\begin{pmatrix}1 & \lambda \\ 0 & 1\end{pmatrix}n_s^{-1} \cdot v$.
    \end{enumerate}
      
    Recall that $\tau$ is the Iwahori model of the universal supersingular representation of $\pi_0$. Let us compute the action of these operators on the basis vectors in $\tau^{I(1)}$ (see Theorem \ref{pro-p-iwahori-invariants}):
    \[[\id, 1],~[\beta, 1], ~s_n^{p^l},~ \beta s_n^{p^l},\] where $n \geq 2$, $0 \leq l \leq f - 1$.    
    \begin{enumerate}
        \item \underline{The operator $T_{\beta}$:} \\
        Using the formula given above, $T_{\beta}$ acts as left-multiplication by $\beta^{-1}$. This means that $T_{\beta}$ flips $[\id, 1], [\beta, 1]$ and flips $s_n^{p^l}, \beta s_n^{p^l}$.

        \item \underline{The idempotent $e_{\chi}$ corresponding to a character $\chi : H \to \Fq^{\times}$:} \\
        To compute the action of $e_{\chi}$, we first compute the actions of $h^{-1}$ on all the basis vectors for all $h \in H$. It can be easily seen that $H$ acts trivially on $[\id, 1]$ and $[\beta, 1]$. Moreover, $h^{-1} = \diag([\lambda_1^{-1}], [\lambda_2^{-1}])$ acts on $s_n^{p^l}$ as multiplication by $(\lambda_2/\lambda_1)^{p^l}$. Indeed, $h^{-1}g^0_{n, \mu} = g^0_{n, \mu'}h^{-1}$, where the digits of $\mu'$ are the digits of $\mu$ each multiplied by $[\lambda_2/\lambda_1]$. The transformations $\mu_i \mapsto \mu_i \cdot (\lambda_1/\lambda_2)$ then prove that $h^{-1}s_n^{p^l} = (\lambda_1/\lambda_2)^{p^l}s_n^{p^l}$. For the action on $\beta s_n^{p^l}$, we see that $\beta^{-1} h^{-1} \beta = \diag([\lambda_2^{-1}], [\lambda_1^{-1}])$. Using this, we immediately see that $h^{-1}$ acts on $\beta s_n^{p^l}$ as multiplication by $(\lambda_2/\lambda_1)^{p^l}$.

        Using these computations with the fact that the sum of values of a multiplicative character of a finite group is $0$ if the character is non-trivial and is the order of the group if the character is trivial, we may write down the action of $e_{\chi}$ on the basis vectors. The actions are
        \begin{eqnarray*}
            &[\id, 1]\cdot e_{\chi} =
            \begin{cases}
                0 & \text{ if } \chi \neq 1 \\
                [\id, 1] & \text{ if } \chi = 1,
            \end{cases}&
            \qquad
            [\beta, 1]\cdot e_{\chi} =
            \begin{cases}
                0 & \text{ if } \chi \neq 1 \\
                [\beta, 1] & \text{ if } \chi = 1,
            \end{cases} \\
            &s_n^{p^l}\cdot e_{\chi} = 
            \begin{cases}
                0 & \text{ if } \chi \neq (d/a)^{p^l} \\
                s_n^{p^l} & \text{ if } \chi = (d/a)^{p^l},
            \end{cases}&
            \qquad
            \beta s_n^{p^l} \cdot e_{\chi} = 
            \begin{cases}
                0 & \text{ if } \chi \neq (a/d)^{p^l} \\
                \beta s_n^{p^l} & \text{ if } \chi = (a/d)^{p^l}.
            \end{cases}
        \end{eqnarray*}

        \item \underline{The operator $T_{n_s}$:} \\
        The action of $T_{n_s}$ is slightly difficult to compute. We first make $T_{n_s}$ act on $[\id, 1]$:
        \begin{eqnarray*}
            [\id, 1] \cdot T_{n_s} & = & \sum_{\lambda \in I_1}\left[\begin{pmatrix}1 & \lambda \\ 0 & 1\end{pmatrix}\begin{pmatrix}0 & 1 \\ -1 & 0\end{pmatrix}, 1\right] \\
            & = & \sum_{\lambda \in I_1}\left[\begin{pmatrix}1 & 0 \\ 0 & -1\end{pmatrix}\beta \begin{pmatrix}1 & 0 \\ -\varpi \lambda & \varpi\end{pmatrix}, 1\right].
        \end{eqnarray*}
        As $\lambda$ runs through all elements in $I_1$, so does $-\lambda$. Therefore,
        \begin{eqnarray*}
            [\id, 1] \cdot T_{n_{s}} & = & T_{1, 2}\left[\begin{pmatrix}1 & 0 \\ 0 & -1\end{pmatrix}\beta, 1\right] = 0 \mod ({\im} T_{1, 2}, {\Ker}T_{1, 2}).
        \end{eqnarray*}
        Then, we make $T_{n_s}$ act on $[\beta, 1]$:
        \begin{eqnarray*}
            [\beta, 1] \cdot T_{n_s} & = & \sum_{\lambda \in I_1}\left[\begin{pmatrix}1 & \lambda \\ 0 & 1\end{pmatrix}\begin{pmatrix}0 & 1 \\ -1 & 0\end{pmatrix} \beta, 1\right] \\
            & = & \sum_{\lambda \in I_1}\left[\begin{pmatrix}\varpi & \lambda \\ 0 & 1\end{pmatrix}\begin{pmatrix}1 & 0 \\ 0 & -1\end{pmatrix}, 1\right] \\
            & = & T_{-1, 0}[\id, 1] \equiv -T_{1, 0}[\id, 1] = -[\beta, 1].
        \end{eqnarray*}
        The congruence above is modulo $({\im} T_{1, 2}, {\Ker} T_{1, 2})$ and obtained by Lemma \ref{L1}(ii).

        Then, the action on $\beta s_n^{p^l}$:
        \begin{eqnarray*}
            \beta s_n^{p^l}\cdot T_{n_s} & = & \sum_{\lambda \in I_1}\sum_{\mu \in I_n}\mu_{n - 1}^{p^l}\left[\begin{pmatrix}1 & \lambda \\ 0 & 1\end{pmatrix}\begin{pmatrix}0 & 1 \\ -1 & 0\end{pmatrix} \begin{pmatrix}0 & 1\\ \varpi & 0\end{pmatrix}\begin{pmatrix}\varpi^n & \mu \\ 0 & 1\end{pmatrix}, 1\right] \\
            & = & \sum_{\lambda \in I_1}\sum_{\mu \in I_n}\mu_{n - 1}^{p^l}\left[\begin{pmatrix}\varpi^{n + 1} & \lambda - \varpi \mu \\ 0 & 1\end{pmatrix}\begin{pmatrix}1 & 0 \\ 0 & -1\end{pmatrix}, 1\right].
        \end{eqnarray*}
        Making the transformations $\mu_i \mapsto -\mu_i$ for $0 \leq i \leq n - 1$, we see that
        \[
            \beta s_n^{p^l} \cdot T_{n_s} = -s_{n + 1}^{p^l}.
        \]
    Now, we may use the relation $e_{\chi}T_{n_s}^2 = -e_{\chi}e_{\chi^w}T_{n_s}$ given, e.g., in \cite[Lemma 2.0.12]{Pas04} to see that $T_{n_s}$ annihilates $s_n^{p^l}$ for $n \geq 3$. Now, the only computation that is left is the action of $T_{n_s}$ on $s_2^{p^l}$. We write
        \begin{eqnarray*}
            s_2^{p^l} \cdot T_{n_s} & = & \sum_{\lambda \in \Fq, \mu \in I_2} \mu_{1}^{p^l}\left[\begin{pmatrix}1 & [\lambda]  \\ 0 & 1\end{pmatrix}\begin{pmatrix}0 & - 1 \\ 1 & 0\end{pmatrix}\begin{pmatrix}\varpi^2 & \mu \\ 0 & 1\end{pmatrix}, 1\right] \\
            & = & \sum_{\lambda \in \Fq, \mu \in I_2, \mu_0 \neq 0} - \mu_{1}^{p^l}\left[\begin{pmatrix}\varpi^{2} & \frac{1 + [\lambda]\mu}{\mu} \\ 0 & 1\end{pmatrix}, 1\right] \\
            && \qquad + \sum_{\lambda \in \Fq, \mu \in I_2, \mu_0 = 0}-\mu_{1}^{p^l}\left[\beta \begin{pmatrix}\varpi & \frac{\mu}{\varpi(1 + [\lambda]\mu)} \\ 0 & 1\end{pmatrix}, 1\right].
        \end{eqnarray*}
        We separately show that these two sums are $0$. Consider the first sum. A short computation shows that $\frac{1 + [\lambda]\mu}{\mu} \equiv [\mu_0^{-1}] - [\mu_0^{-2}\mu_1]\varpi + [\lambda] \mod \varpi^2$. Making the change of variables $\mu_0 \mapsto \mu_0^{-1}$, we see that the first sum becomes
        \begin{eqnarray*}
            && \sum_{\lambda \in \Fq, \mu \in I_2, \mu_0 \neq 0} -\mu_1^{p^l}\left[\begin{pmatrix}\varpi^2 & [\mu_0] + [\lambda] - [\mu_0^{2}\mu_1]\varpi \\ 0 & 1\end{pmatrix}, 1\right] \\
            && = \sum_{\lambda \in \Fq, \mu \in I_2, \mu_0 \neq 0}-\mu_1^{p^l}\left[\begin{pmatrix}\varpi^2 & [\mu_0 + \lambda] + [-\mu_0^{2}\mu_1 + Z(\mu_0, \lambda)]\varpi \\ 0 & 1\end{pmatrix}, 1\right],
        \end{eqnarray*}
        where, by Lemma \cite[Lemma 4.1]{Jan23}, we have
        \[
            Z(\mu_0, \lambda) =
            \begin{cases}
                0 & \text{ if } e > 1 \\
                \dfrac{\mu_0^q + \lambda^q - (\mu_0 + \lambda)^q}{p} & \text{ if } e = 1.
            \end{cases}
        \]
        Making the change of variables $\mu_0 \mapsto \mu_0 - \lambda$, $\mu_1 \mapsto -(\mu_0 - \lambda)^{-2}(\mu_1 - Z(\mu_0 - \lambda, \lambda))$, we get
        \[
            \sum_{\lambda \in \Fq, \mu \in I_2, \mu_0 \neq \lambda} (\mu_0 - \lambda)^{-2p^l}(\mu_1 - Z(\mu_0 - \lambda, \lambda))^{p^l}\left[g^0_{2, \mu}, 1\right].
        \]
        Now, making the transformation $\lambda \mapsto \mu_0 - \lambda$, the sum becomes
        \[
            \sum_{\lambda \in \Fq^{\times}, \mu \in I_2} (\lambda)^{-2p^l}(\mu_1 - Z(\lambda, \mu_0 - \lambda))^{p^l}\left[g^0_{2, \mu}, 1\right].
        \]
        Since $g^0_{2, \mu}$ is independent of $\lambda$, we may use the fact that $\sum_{\lambda \in \Fq^{\times}}\lambda^{2p^l} = 0$ for $q > 3$. The largest power of $\mu_0$ in $Z$ is strictly less than $q - 1$. These two facts (with a transformation $\mu_0 \mapsto \mu_0^{1/p^{l}}$) imply that the sum corresponding to $\mu_0 \neq 0$ in $s_2^{p^l}$ vanishes.

        Next, consider the sum corresponding to $\mu_0 = 0$. This case is easier. We note that
        \[
            \begin{pmatrix}
                \varpi & \frac{\mu}{\varpi(1 + [\lambda]\mu)} \\ 0 & 1
            \end{pmatrix} = \begin{pmatrix}
                \varpi & [\mu_1] \\ 0 & 1
            \end{pmatrix}\cdot i
        \]
        for some $i \in I(1)$. So, each function appearing in the sum corresponding to $\mu_0 = 0$ is independent of $\lambda$. Summing over $\lambda$ now yields $0$.
    \end{enumerate}

    These results can be compactified in the following table:
    \begin{center}
    \label{action table}
      \begin{tabular}{|c|c|c|c|}
         \hline
         & $T_{\beta}$ & $T_{n_s}$ & $e_{\chi}$ \\ [0.5ex] 
         \hline
         $[\id, 1]$ & $[\beta, 1]$ & $0$ & $\begin{cases} 0 & \text{ if $\chi$ is non-trivial} \\ [\id, 1] & \text{ if $\chi$ is trivial} \end{cases}$ \\ 
         \hline
         $[\beta, 1]$ & $[\id, 1]$ & $-[\beta, 1]$ & $\begin{cases} 0 & \text{ if $\chi$ is non-trivial} \\ [\beta, 1] & \text{ if $\chi$ is trivial} \end{cases}$ \\
         \hline
         $s_n^{p^l}$ & $\beta s_n^{p^l}$ & 0 & $\begin{cases} 0 & \text{ if $\chi \neq (d/a)^{p^l}$} \\ s_n^{p^l} & \text{ if $\chi = (d/a)^{p^l}$} \end{cases}$ \\
         \hline
         $\beta s_n^{p^l}$ & $s_n^{p^l}$ & $-s_{n + 1}^{p^l}$ & $\begin{cases} 0 & \text{ if $\chi \neq (a/d)^{p^l}$} \\ \beta s_n^{p^l} & \text{ if $\chi = (a/d)^{p^l}$} \end{cases}$ \\
         \hline
        \end{tabular}
        
        \vspace{.1cm}
        Action of $\cH$ on $\tau^{I(1)}$ for $q>3$. 
        \end{center}

    \begin{remark}\label{generator}
        We remark that the subspace of $I(1)$-invariants in $\pi_0$ is generated by the element 
        \[
            [\id, 1] + \sum_{l = 0}^{f - 1}s_2^{p^l}
        \]
        over the pro-$p$-Iwahori Hecke algebra $\cH$. Indeed, applying the idempotent corresponding to the trivial character to this element yields $[\id, 1]$, whereas applying the idempotent corresponding to $(d/a)^{p^l}$ yields $s_2^{p^l}$. The remaining basis vectors mentioned above can then be obtained from $[\id, 1]$ or $s_2^{p^l}$ by successively applying $T_{\beta}$ and $T_{n_s}$.
    \end{remark}
\subsection{Comparison between Iwahori and pro-$p$-Iwahori induction}
    Using the following lemmas, we write down a direct sum decomposition of $\IIZind \mathbbm{1}$ with a factor isomorphic to $\IZind \mathbbm{1}$. Let $\eta: Z \rightarrow \mathbb{F}_q^{\times}$ denote the character sending $a \in \mathcal{O}^{\times}$ to $a$ and $\varpi$ to $1$.
    \begin{lemma}\label{Direct sum decomp}
         For $0 \leq r, s \leq q - 2$, the map
        \begin{eqnarray*}
            a^{r - s}d^s & \to & \IIZIZind \eta^{r} \\
            v & \mapsto & -\sum_{\lambda \in \Fq^{\times}}\lambda^{-s}\left[\begin{pmatrix}1 & 0 \\ 0 & [\lambda]\end{pmatrix}, 1\right]
        \end{eqnarray*}
        is $IZ$-equivariant, where $v$ is a basis element of $a^{r - s}d^s$. Moreover, putting all of them together, we get an isomorphism
        \[
            \bigoplus_{s = 0}^{q - 2} a^{r - s}d^s \xrightarrow{\sim} \IIZIZind \eta^r.
        \]
    \end{lemma}
    \begin{proof}
        We only need to show that $I$ acts via $a^{r - s}d^s$ on the element
        \[
            f = -\sum_{\lambda \in \Fq^{\times}}\lambda^{-s}\left[\begin{pmatrix}1 & 0 \\ 0 & [\lambda]\end{pmatrix}, 1\right].
        \]
        Indeed, using the identity
        \[
            \begin{pmatrix}
                a & b \\ \varpi c & d
            \end{pmatrix}
            \begin{pmatrix}
                1 & 0 \\ 0 & [\lambda]
            \end{pmatrix} = 
            \begin{pmatrix}
                1 & 0 \\ 0 & [d\lambda/a]
            \end{pmatrix}
            \begin{pmatrix}
                a & b[\lambda] \\ \frac{c\varpi}{[d\lambda/a]} & \frac{d}{[d/a]}
            \end{pmatrix},
        \]
        we see that
        \[
            \begin{pmatrix}
                a & b \\ \varpi c & d
            \end{pmatrix}f = -\sum_{\lambda \in \Fq^{\times}}\lambda^{-s}a^{r}\left[\begin{pmatrix}1 & 0 \\ 0 & [d\lambda/a]\end{pmatrix}, 1\right].
        \]
        The transformation $\lambda \mapsto a\lambda/d$ then yields
        \[
            \begin{pmatrix}
                a & b \\ \varpi c & d
            \end{pmatrix}f = a^{r - s}d^s f.
        \]

        For the final claim, we note that both sides are $q-1$ dimensional and that the map is surjective since
        \begin{eqnarray*}
            \sum_{s = 0}^{q - 2}-\sum_{\lambda \in \Fq^{\times}}\lambda^{-s}\left[\begin{pmatrix}1 & 0 \\ 0 & [\lambda]\end{pmatrix}, 1\right] = \sum_{\lambda \in \Fq^{\times}}-\left(\sum_{s = 0}^{q - 2}\lambda^{-s}\right)\left[\begin{pmatrix}1 & 0 \\ 0 & [\lambda]\end{pmatrix}, 1\right] = [\id, 1]. 
        \end{eqnarray*}
        This proves the direct sum decomposition of $\IIZIZind \eta^r$ claimed in the statement.
    \end{proof}

    Recall that the Hecke algebra $\cH$ is generated by $T_{\beta}, T_{n_s}$ and $e_{\chi}$ for $\chi : H \to \Fq^{\times}$. We write down formulas for the action of these Hecke operators on the space $\IIZind \mathbbm{1}$:
    \begin{enumerate}
        \item The formula for $T_{\beta}:$ 
        \[
            T_{\beta}[g, 1] = [g\beta, 1].
        \]
        \item The formula for $T_{n_s}:$
        \[
            T_{n_s}[g, 1] = \sum_{\lambda \in I_1}\left[g\begin{pmatrix}1 & \lambda \\ 0 & 1\end{pmatrix}\begin{pmatrix}0 & 1 \\ -1 & 0\end{pmatrix}, 1\right].
        \]
        \item The formula for $e_{\chi}$ for a character $\chi : H \to \Fq^{\times}:$
        \[
            e_{\chi}[g, 1] = \sum_{h \in H} \chi(h)\left[gh^{-1}, 1\right].
        \]
    \end{enumerate}

    \begin{lemma}\label{Comparison theorem}
        Let $\chi : H \to \Fq^{\times}$ be a character which agrees with $\eta^r$ on $Z\cap H$. Then,
        \begin{eqnarray*}
            \frac{\IIZind \eta^r}{\oplus_{\chi' \neq \chi}{\im} e_{\chi'}} & \to & \IZind \chi \\
            \left[\id, 1\right] & \mapsto & [\id, 1]
        \end{eqnarray*}
        is an isomorphism. Furthermore, the same map induces an isomorphism
        \[
            \frac{\IIZind \eta^r}{{\im} (1 - e_{\chi})} \simeq \IZind \chi.
        \]
        Moreover, under this isomorphism, $T_{\beta}$ and $T_{n_s}$ on the left get mapped to $T_{1, 0}$ and $T_{1, 2}T_{1, 0}$ on the right, respectively.
    \end{lemma}
    \begin{proof}
        The central character of $\IIZind \eta^r$ is $\eta^r$. The image of the idempotent $e_{\chi'}$ consists of elements on which $Z \cap H$ acts by $\chi'$. Therefore if $\chi'$ does not agree with $\eta^r$ on $Z\cap H$, then ${\im} e_{\chi'} = 0$. Note that the remaining $e_{\chi}$ are complementary idempotents. Indeed, writing $h\in H$ as $\begin{pmatrix}[h_1] & 0 \\ 0 & [h_2] \end{pmatrix}$, we have
        \[ \sum_{s = 0}^{q - 2}(e_{a^{r - s}d^s})[g, 1] =  \sum_{h \in H} \sum_{s = 0}^{q - 2}h_1^{r - s}h_2^s[gh^{-1}, 1]  =  \sum_{h \in H}h_1^r\sum_{s = 0}^{q - 2}(h_2/h_1)^s[gh^{-1}, 1] = -\sum_{h \in Z \cap H}h_1^r[gh^{-1}, 1].\]
        The last expression clearly equals $[g, 1]$.
        This means that we have two direct sum decompositions of $\IIZind \eta^r$. We claim that under the map obtained using Lemma~\ref{Direct sum decomp}, $\IZind \chi$ maps to ${\im} e_{\chi}$ for $\chi = a^r, a^{r - 1}d, \ldots, a^{r - q + 2}d^{q - 2}$. To see this, for $\chi = a^{r-s}d^s$, we show that
        \[
            e_{\chi}[\id, 1] = -\sum_{\lambda \in \Fq^{\times}}\lambda^{-s}\left[\begin{pmatrix}1 & 0 \\ 0 & [\lambda]\end{pmatrix}, 1\right].
        \]
        Indeed,
        \begin{eqnarray*}
            e_{\chi}[\id, 1] & = & \sum_{h_1, h_2 \in \Fq^{\times}}\chi\begin{pmatrix}[h_1] & 0 \\ 0 & [h_2]\end{pmatrix}\left[\begin{pmatrix}[h_1^{-1}] & 0 \\ 0 &[ h_2^{-1}]\end{pmatrix}, 1\right] \\
            & = & \sum_{h_1, h_2 \in \Fq^{\times}} (h_1/h_2)^{-s}\left[\begin{pmatrix}1 & 0 \\ 0 & [h_1/h_2]\end{pmatrix}, 1\right] = -\sum_{\lambda \in \Fq^{\times}}\lambda^{-s}\left[\begin{pmatrix}1 & 0 \\ 0 & [\lambda]\end{pmatrix}, 1\right].
        \end{eqnarray*}
        
        Next, using the direct sum decomposition in Lemma~\ref{Direct sum decomp}, we see that ${\im} e_{a^{r - s}d^s}$ is isomporphic to $\IZind a^{r - s}d^s$. So, once we mod out by the idempotent corresponding to every $\chi' \neq \chi$, we get $\IZind \chi$. 
    \end{proof}
\subsection{Proof of Theorem \ref{summand-in-reverse-vigneras-functor-intro}}
    We begin by describing the annihilator of the $\cH$-submodule of $\pi_0^{I(1)}$ generated by $[\id, 1]$ as a right ideal in the non-commutative algebra $\cH$.
    \begin{lemma}\label{Annihilator lemma}
        The annihilator of the $\cH$-submodule of $\pi_0^{I(1)}$ generated by $[\id, 1]$ is the right ideal generated by $T_{n_s}, T_{\beta}(T_{n_s} + 1)$, and $e_{\chi}$ for $\chi \neq \mathbbm{1}$.
    \end{lemma}
    \begin{proof}
      First observe that the operators $T_{n_s}, T_{\beta}(T_{n_s} + 1)$, and $e_{\chi}$ for $\chi \neq \mathbbm{1}$ annihilate $[\id, 1]$ by the table given at the end of \S\ref{action-on-I(1)-invariants}. Let $J$ be the right $\cH$-ideal generated by $T_{n_s}, T_{\beta}(T_{n_s} + 1)$, and $e_{\chi}$ for $\chi \neq \mathbbm{1}$. Note that $e_{\mathbbm{1}} - 1$ also belongs to $J$. Suppose $\phi \in \cH$ annihilates the $\cH$-submodule of $\pi_0^{I(1)}$ generated by $[\id, 1]$. Using the identities (see \cite[Lemma 2.0.12]{Pas04}) 
        \[
            T_{n_s}e_{\chi} = e_{\chi^{s}}T_{n_s}, T_{\beta}e_{\chi} = e_{\chi^{s}}T_{\beta}, \text{ and } e_{\chi}e_{\chi'} = \begin{cases}
                e_{\chi} & \text{ if } \chi = \chi' \\
                0 & \text{ if } \chi \neq \chi',
            \end{cases}
        \]
        we may write
        \[
            \phi = P_0'(T_{\beta}, T_{n_s}) + \sum_{\chi \>:\> H \to \Fq^{\times}} e_{\chi}P_{\chi}'(T_{\beta}, T_{n_s})
        \]
        for some non-commutative polynomials $P_0'$ and $P_{\chi}'$. Since $J$ contains $e_{\chi}$ for $\chi \neq 1$ and $e_\mathbbm{1}-1$, we see that $\phi$ is congruent to a non-commutative polynomial $P_0(T_{\beta}, T_{n_s})$ in $T_{\beta}$ and $T_{n_s} \mod J$.

        Next, we claim that modulo $J$, we may write $\phi$ as an $\overline{\F}_p$-linear combination of the identity map and $T_{\beta}$. Indeed, let $m$ be a monomial in $P_0$. If the length of $m$ (the sum of powers of $T_{\beta}$ and $T_{n_s}$ appearing in $m$) is greater than or equal to $3$, then we may use the congruences $T_{n_s} \equiv 0$ and $T_{\beta}T_{n_s} \equiv - T_{\beta}$ modulo $J$ to reduce the length of $m$ by at least $1$. This means that $P_0$ is a linear combination of $1, T_{\beta}, T_{\beta}^2$ and $T_{\beta}T_{n_s}$ modulo $J$. The claim is then proved using the relations $T_{\beta}^2 = 1$ and $T_{\beta}T_{n_s} \equiv - T_{\beta}$ modulo $J$. Therefore, we may write
        \[
            \phi \equiv a + bT_{\beta} \mod J.
        \]
        Applying this to $[\id, 1]$, and noting that $[\id, 1]$ and $[\beta, 1]$ are linearly independent, we see that $\phi \in J$.
    \end{proof}

  Recall that $\pi_0 \simeq \frac{\IZind \mathbbm{1}}{({\im} T_{1, 2}, {\Ker} T_{1, 2})}$. Next, we compute the image of $\pi_0^{I(1)}$ under the reverse functor studied by Vign\'eras, Ollivier and others, and prove Theorem \ref{summand-in-reverse-vigneras-functor-intro} in the introduction. 
\begin{theorem}\label{summand-in-reverse-vigneras-functor}
   Let $\pi_0$ denote the universal supersingular representation of $G$. Then there exists a representation $M$ of $G$ such that
   \[
   \pi_0^{I(1)} \otimes_{\cH} \IIZind \mathbbm{1} \simeq \pi_0 \oplus M.
   \]
\end{theorem}

\begin{proof}
Using $\IIZind \mathbbm{1} \simeq \bigoplus_{\chi} {\im}e_{\chi}$, we see that
    \[
        \pi_0^{I(1)} \otimes_{\cH} \IIZind \mathbbm{1} \simeq \left([\id, 1]\cH \otimes_{\cH} \IIZind \mathbbm{1}\right) \oplus M
    \]
    for some representation $M$. Moreover, it can be seen that the image of the composition of $\IZind \mathbbm{1} \hookrightarrow \IIZind \mathbbm{1} \twoheadrightarrow \pi_0^{I(1)} \otimes_{\cH} \IIZind \mathbbm{1}$ is exactly $[\id, 1]\cH \otimes_{\cH} \IIZind \mathbbm{1}$, which is further isomorphic to the quotient of $\IIZind \mathbbm{1}$ by images of the Hecke operators belonging to the annihilator of $[\id, 1]\cH$.
     Finally, using Lemma~\ref{Annihilator lemma} and Lemma~\ref{Comparison theorem}, we see that 
    \[
        [\id, 1]\cH \otimes_{\cH} \IIZind \mathbbm{1} \simeq \frac{\IIZind \mathbbm{1}}{({\im} T_{n_s}, {\im} T_{\beta}(T_{n_s} + 1), \oplus_{\chi \neq \mathbbm{1}} {\im} e_{\chi})} \simeq \pi_0.
    \]
    This shows that $\pi_0$ is a direct summand of $\pi_0^{I(1)} \otimes_{\cH} \IIZind \mathbbm{1}$.
\end{proof}
    \section*{Acknowledgements}
        The first author is supported by the National Research Foundation of Korea (NRF) grant funded by the Korea government (MSIT) (No. RS-2025-02262988 and No. RS-2025-00517685). The second author would like to thank the Kerala School of Mathematics (KSoM) for its support and excellent research environment. The third author acknowledges the financial assistance as a Junior Research Fellowship (NET) from UGC, India.
\printbibliography
    
\end{document}